\documentclass[10pt]{amsart}
\usepackage{amsmath} 
\numberwithin{equation}{section}
\usepackage{amssymb,amsmath,amsthm}
\usepackage{microtype}
\usepackage{slashed}
\usepackage{xcolor}
\usepackage[hyperfootnotes=false]{hyperref}
\hypersetup{linktoc = all}                
\hypersetup{pdfborderstyle={/S/U/W 0.5}}  
\hypersetup{hidelinks}
\hypersetup{bookmarksnumbered}
\hypersetup{
  pdftitle={A nonlinear rod linkage model: zero-energy chains and continuum consistency},
  pdfauthor={Dennis Kriventsov and Philippe G. LeFloch}
}

\makeatletter
\renewcommand{\@settitle}{%
  \begin{center}%
    \baselineskip14\p@\relax
    \bfseries\@title
  \end{center}%
}
\makeatother

\usepackage{graphicx}
\usepackage{subcaption}
\usepackage{placeins}
\usepackage{needspace}
\usepackage{tikz,pgfplots}
\pgfplotsset{compat=1.12}
\usepackage{standalone}
\usepackage{stmaryrd}

\usepackage{amssymb,amsthm}
\usepackage{esint}
\usepackage{epstopdf}
\usepackage{color}
\usepackage[normalem]{ulem}
\usepackage{cancel}
\usepackage{url}
\usepackage[top=1in, left=1in, right=1in, bottom=1in]{geometry}

\newtheorem{theorem}{Theorem}[section]
\newtheorem{lemma}[theorem]{Lemma}
\newtheorem{proposition}[theorem]{Proposition}

\newtheorem{question}[theorem]{Question}
\newtheorem{conjecture}[theorem]{Conjecture}
\newtheorem{definition}{Definition}[section]

\definecolor{revisionviolet}{rgb}{0.48,0.08,0.66}

\def\Z{{\mathbb Z}}
\def\R{{\mathbb R}}

\def\a{\alpha}
\def\b{\beta}

\def\d{\delta}

\def\n{\nabla}
\def\p{\partial}

\def\s{\sigma}
\def\t{\tau}

\def\g{\gamma}

\def\1{\left(}
\def\2{\right)}
\def\3{\left\{}
\def\4{\right\}}
\def\8{\infty}

\newcommand \bse {\begin{subequations}}
\newcommand \ese {\end{subequations}}

\newcommand \Kt 		{\widetilde K}
\newcommand \thetab	{\overline\theta}
\newcommand \bl   		{\overline{l}}

\newcommand \be {\begin{equation}}
\newcommand \bel {\be\label}
\newcommand \ee {\end{equation}}
 
\newcommand \RR {\mathbb R} 
\newcommand \rt {\widetilde r} 
\newcommand \lt {\widetilde l} 
\newcommand \Acal {\mathcal A}
\newcommand \del \partial

\newcommand \bei {\begin{itemize}}
\newcommand \eei {\end{itemize}}

\newcommand \Lbf {\mathbf L}
\newcommand \Gbf {\mathbf G}

\begin{document}
   
\title[A discrete rod-linkage model]{\Large Zero-energy standing waves and continuum limits
\\
for a discrete rod-linkage model}

\author[D. Kriventsov and P.G. L{\tiny e}Floch]{Dennis Kriventsov$^1$ and Philippe G. L{\scriptsize e}Floch$^2$}

\begin{center} 
{  \sl In memory of Robert V. Kohn$^*$}
\end{center}

\begin{abstract}
We study a discrete Lagrangian model for a one-dimensional chain of rigid rods coupled by elastic springs. We characterize the geometry of its zero-energy constraint and distinguish the spinner and flipper regimes. Under graph-type hypotheses, selected zero-energy chains converge, after rescaling, to trajectories of an autonomous ordinary differential equation. For the physical linkage, we identify a simple alternating selection criterion under which chains form a unique monotone heteroclinic family in the spinner regime. In the flipper regime, we determine the topology of the constraint set and the number of successors, but leave global branch selection open. We also derive one- and two-component continuum residual expansions, distinguishing consistency from convergence of time-dependent solutions. Reproducible numerical experiments test the selected spinner recurrence and its asymptotic multiplier; for a simplified one-component model with smooth periodic data, they recover first- and second-order discrepancies for the leading and corrected continuum equations, respectively.
\end{abstract}

\footnotetext{$^1$ Department of Mathematics, Rutgers University, 110 Frelinghuysen Road, Piscataway, NJ 08854-8019, USA. Email: {\em dnk34@math.rutgers.edu.}
\newline 
$^2$
Laboratoire Jacques-Louis Lions \& Centre National de la Recherche Scientifique, 
Sorbonne Universit\'e, 4 Place Jussieu, 75252 Paris, France. Email: {\em contact@philippelefloch.org.}
%
\newline 
$^*$ \it The authors are very grateful to Bob Kohn for suggesting this problem and for many stimulating discussions. Bob had a unique ability to bridge the gap between communities in physics and materials science on one hand and nonlinear PDE and calculus of variations on the other hand, translating into rigorous, simple, mathematics ideas and experiments that seemed to defy our abilities to even describe in that language. His work spanned some of the most complex and nonlinear phenomena, including fracture, wrinkling, and fluids, and opened the way to their mathematical study. Attending his seminar was a formative experience, and broadened one's perspective of what mathematics is really about and what kind of knowledge it can create. His presence in our community is sorely missed.
}
 \maketitle 



\section{Introduction and main results}
\label{sec:1} 

\subsection{Scope and questions}

Mechanical linkages provide elementary examples in which a finite-dimensional geometric constraint produces nonlinear waves on a large lattice. Kane and Lubensky \cite{KL} introduced a rotor chain, then studied by Chen, Upadhyaya, and Vitelli \cite{CUV}, which is especially instructive: its linearized mechanics has a polarized edge mode, whereas its finite-amplitude motion transports a \emph{zero-energy deformation} through the chain. Related work also has emphasized the asymmetry of kinks and antikinks and the role of the discrete geometry in their propagation \cite{ZhouChenUpadhyayaVitelli}. In \cite{ST, L} the discrete lattice model's zero-energy modes and dynamics are considered further. The work \cite{Berry} examines the question of constructing similar mechanical systems with prescribed structure.

The mathematical problem of interest here is to derive a \emph{continuum equation}, but more importantly is also to determine \emph{which discrete branches} that equation represents and \emph{on which space--time} scale the representation is valid.

The same distinction appears in the theory of phase boundaries. A discrete model may select a transition that is not determined by the leading continuum equation alone, and a higher-order correction may affect the observed profile before the asymptotic regime is reached; see, for example, \cite{AbeyaratneKnowles1,HayesLeFloch,LeFloch-ARMA,LeFlochMohamadian}. In the present linkage setup, the leading equation is of nonlinear Klein--Gordon type rather than a system of conservation laws. Nevertheless, the central questions are analogous: existence of connecting profiles, branch selection, stability, and quantitative comparison of discrete and continuum dynamics.


The analysis has four parts.
\bei

\item We derive the discrete Lagrangian equation and determine the geometry of its zero-energy constraint.

\item We construct and classify selected zero-energy standing configurations.

\item We derive residual expansions for one- and two-component continuum fields under stated smoothness and scaling hypotheses.

\item We compare the discrete and continuum models numerically on common grids and with common initial data.
\eei
 

\subsection{Main results}

The main results are as follows.
\begin{enumerate}
\item Under hypotheses~\textup{(H1)--(H6)} in Section~\ref{sec:1periodic}, every initial angle determines a unique graph-type zero-energy chain. Whenever the normalized recurrence converges, the piecewise-linear interpolants converge to the solution of an autonomous ordinary differential equation; see Theorem~\ref{thm:one-periodic-limit}. The preceding classification lemma identifies the endpoint equilibria of every nonconstant chain.
\item For the physical linkage, Theorem~\ref{thm:spinners} constructs a unique monotone heteroclinic chain for each admissible initial angle in the spinner regime, within a prescribed alternating component class. In the flipper regime, Theorem~\ref{prop:flipper-geometry} determines the topology of the zero set and the number of successors, including the two fold points. Global branch selection remains open and is formulated in Conjecture~\ref{conj:flipper-selection}.
\item For the simplified one-component model, Propositions~\ref{prop:first-order-consistency} and \ref{prop:second-order-consistency} give the leading Klein--Gordon equation and its first correction as local residual expansions. For two-component configurations, Section~\ref{sec:4} gives the limiting branch offsets \(\pi\pm2\thetab\), the exact even--odd lattice system, and its Taylor residual. These consistency results do not establish convergence of time-dependent lattice solutions.
\item The numerical experiments in Section~\ref{sec:numerics} test the selected spinner recurrence and multiplier, the local residual orders, and the finite-time discrepancies between the discrete and continuum evolutions for smooth periodic data. They provide reproducible evidence in the stated parameter ranges, but do not prove discrete-to-continuum convergence or stability.
\end{enumerate}
Orbital stability, long-time error control, and flipper branch selection remain open.


\subsection{Organization of this paper}

Section~\ref{sec:2} introduces the physical model, its rescaling, and the parameter regimes. Section~\ref{sec:3} treats abstract zero-energy chains and continuum consistency. Section~\ref{sec:4} develops the two-component formulation for the physical linkage. Section~\ref{sec:5} proves the spinner-chain theorem and analyzes the flipper constraint and the open branch-selection problem. Section~\ref{sec:numerics} reports the numerical experiments.


\section{Rod-linkage geometry, energy, and scaling}
\label{sec:2} 

\subsection{Discrete geometry and variational formulation}

Our analysis and terminology follows that of \cite{CUV}, which already observes some of the key structure of the model.

Fix four physical parameters \(a>0\), \(r>0\), \(\thetab\in(0,\pi/2)\), and \(K>0\). The linkage is defined as follows.

\bei 

\item In the plane, at each point with coordinates $(a k, 0) \in \RR^2$ (with $k\in \Z$) we \emph{anchor one end of a massive rigid rod} of length $r$, which is otherwise \emph{free to rotate} about that anchor. 

\item In the \emph{reference configuration,} the rod at an even site has \emph{angle \(\thetab\)} from the positive \(y\)-axis, while the rod at an odd site has \emph{angle \(\pi-\thetab\)}. With this convention, the free vertex of every reference rod \emph{lies to the right of its pivot.}

\item The free vertices of every adjacent pair are connected by a \emph{massless Hookean spring} of stiffness \(K\), relaxed in the reference configuration. The moment of inertia of each rod is normalized to one. 

\eei 

We denote by \(\bl\) the \emph{relaxed spring length}, which can be determined from $a, r,$ and $\thetab$: a simple trigonometric calculation shows that
\bel{eq:ltilde}
\bl^2 = a^2 + 4r^2 \cos^2 \thetab.
\ee  
The discrete dynamical system of interest can then be fully described in terms of the angle $\theta_k$ of the rod anchored at each point $(ak,0)$: 
the {\em kinetic energy} of each rod is defined to be 
$
 \frac{1}{2} \dot{\theta}_k^2, 
$
while the {\em potential energy of the spring joining sites \(k\) and \(k+1\)} is taken to be 
$
\frac{K}{2}(l_{k,k+1}- \bl)^2, 
$
where $l_{k,k+1}$ denotes the {\sl length of the spring} connecting the $k$-th and $(k+1)$-th rods. Hence, restricting our attention to the rods with fixed vertices between the origin $(0,0)$ and the point $(aN,0)$ for some fixed integer $N$, we introduce the \textbf{finite-chain Lagrangian}
\be
\mathcal L_N := 
 \frac{1}{2} \sum_{k=0}^N \dot \theta_k^2 -  \frac{K}{2} \sum_{k=0}^{N-1} (l_{k,k+1} - \bl)^2,
\ee
while the corresponding action on a time interval \([0,T]\) is
\be
\Acal_N[\theta]:=\int_0^T\mathcal L_N(\theta(t),\dot\theta(t))\,dt.
\ee
Moreover, $l_{k,k+1}$ may be expressed in terms of $\theta_k$ and $\theta_{k+1}$ directly. We write \(\Lbf_{a,r}(\theta_k,\theta_{k+1})=l_{k,k+1}\) for the physical spring length:
\be
\aligned
\Lbf_{a,r}(\theta_k,\theta_{k+1})^2  
& = r^2 \Big( (\cos \theta_k - \cos \theta_{k+1})^2 + (a/r + \sin \theta_{k+1} - \sin \theta_k)^2 \Big)
\\ 
& = a^2 + 4a r\sin\Big(\frac{\theta_{k+1}- \theta_k}{2} \Big)\cos\Big(\frac{\theta_{k+1}+\theta_k}{2} \Big) + 4r^2\sin^2 \Big(\frac{\theta_{k+1}- \theta_k}{2} \Big).
\endaligned
\ee
The variational equation associated with this Lagrangian is the discrete evolution studied below.

\subsection{Zero-energy states and parameter regimes}

A static configuration has zero potential energy precisely when every spring has its relaxed length. Before discussing connecting configurations, we identify the constant and period-two states exactly.

\begin{proposition}[Period-two zero-energy states]\label{lem:period-two-states}
There is no constant zero-energy configuration. The ordered pairs generating period-two zero-energy configurations are exactly
\bel{eq:period-two-states}
(\thetab,\pi-\thetab),\qquad
(\pi-\thetab,\thetab),\qquad
(-\thetab,\pi+\thetab),\qquad
(\pi+\thetab,-\thetab)
\ee
modulo \(2\pi\). Thus there are two configurations up to a translation by one lattice site.
\end{proposition}

\begin{proof}
\bse
A period-two configuration generated by \((x,y)\) is stress free if
\(\Lbf_{a,r}(x,y)=\Lbf_{a,r}(y,x)=\bl\). Subtracting the squares of these two equalities and using the formula above gives
\be
4ar(\sin y-\sin x)=0.
\ee
Hence \(y=x\) or \(y=\pi-x\) modulo \(2\pi\). The first alternative gives \(\Lbf_{a,r}(x,x)=a<\bl\) and is impossible because \(\thetab\in(0,\pi/2)\). In the second alternative, we have 
\be
\Lbf_{a,r}(x,\pi-x)^2=a^2+4r^2\cos^2x.
\ee
Comparison with \(\bl^2=a^2+4r^2\cos^2\thetab\) yields \(\cos^2x=\cos^2\thetab\), and therefore precisely the four pairs in \eqref{eq:period-two-states}.
\ese
\end{proof}

Introduce the dimensionless parameter
\bel{eq-scale-d}
d:=\frac{2r}{a}\sin\thetab, 
\ee
referred to as \textbf{spinner-flipper scale}, whose role is now explained. 
The four period-two states in Proposition~\ref{lem:period-two-states} persist across the regime boundary. The parameter $d$ instead determines the topology of the zero set and the geometry of admissible connecting chains.
\bei
\item In the \emph{spinner range}, defined as 
\be
d>\frac{2}{\sin\thetab},
\ee
the zero set has \emph{two graph components}. The alternating component choice used in Theorem~\ref{thm:spinners}, below, gives a \emph{unique heteroclinic chain} for each admissible initial angle.

\item In the \emph{flipper range},
\be
d<\frac{2}{\sin\thetab},
\ee
the zero set is \emph{one null-homotopic embedded circle with two fold points}; see Theorem~\ref{prop:flipper-geometry}. A one-step successor is generally multivalued.

\eei
At the \textbf{spinner--flipper threshold} \(d=2/\sin\thetab\), the zero set is \emph{singular}: unrestricted switching between local branches also produces zero-energy chains, but \emph{no complete classification} of all such sequences is given in this paper.

In addition, within the flipper range, the value \(d=1\) separates the so-called \emph{non-wobbling} and \emph{wobbling} geometries described in \cite{CUV}. We do not use this notion in the present paper, except to state Conjecture~\ref{conj:flipper-selection}, below.


\subsection{Nondimensionalization and time scale}

We now identify the joint parameter scaling needed to compare the discrete dynamics with continuum equations. We proceed formally and any question about the stability of the stationary configurations is kept for future investigations. 
To investigate continuum limits, we take the cell size parameter
\be
a \rightarrow 0,
\ee
while at the same time selecting appropriate asymptotic behavior for the other parameters $r, \thetab, K$ as functions of $a$. It is convenient to normalize our parameters with $a$. With the physical time variable kept unchanged, the coefficient multiplying the dimensionless spring energy is \(Ka^2\). We therefore set
\bel{eq-rescal-coeff}
\rt = \frac{r}{a}, 
\qquad \Kt = K a^2,
\qquad \lt = \frac{\bl}{a}. 
\ee  
We introduce the continuum time \(T=qt\) by setting \(\theta_k(t)=\theta(x,T)\), with \(x=ka\). The coefficient in the equation written in \(T\) is \(\Kt/q^2\); thus each continuum limit specifies both the material scaling and the time scaling. We henceforth write \(\Lbf:=a^{-1}\Lbf_{a,r}\) for the dimensionless spring length.
It is convenient to regard $\rt$ and $\thetab$ as our \textbf{primary parameters} and express the other parameters $d$ and $\lt$ as functions of $\rt, \thetab$. That is, from \eqref{eq:ltilde} and \eqref{eq-scale-d}, we find 
\be
\lt = \sqrt{1+ 4 \rt^2 \cos^2\thetab}
 = \Lbf(\thetab,\pi-\thetab),  \qquad d = 2\rt \sin \thetab.
\ee  

With this notation and these conventions, we define the \textbf{rescaled finite-chain Lagrangian} as
\bel{eq:action}
\mathcal L_N = \sum_{k=0}^N  \frac{1}{2} \dot \theta_k^2 
        - \frac{\Kt}{2} \sum_{k=0}^{N-1} \big( \Lbf(\theta_k, \theta_{k+1}) - \lt \big)^2
\ee
with
\bel{equa--214}
\Lbf(\theta_1, \theta_2)^2 = 
1 + 4 \rt \sin \Big(\frac{\theta_2 - \theta_1}{2} \Big)\cos \Big(\frac{\theta_2 + \theta_1}{2} \Big) 
+ 4 \rt^2\sin^2 \Big(\frac{\theta_2 - \theta_1}{2} \Big). 
\ee 
We will sometimes use the shorthand notation 
\be
\Gbf(\theta, \omega) := \frac{\Kt}{2}  (\Lbf(\theta, \omega) - \lt)^2, 
\qquad \theta, \omega \in \RR. 
\ee
\needspace{6\baselineskip}
When \(2r\geq a\), the spring length can vanish for some angle pairs: for example, \(y=-x\) and \(2r\sin x=a\). Since \(\bl>0\), the Hookean potential is not differentiable at such a collision. The Euler--Lagrange equation below is therefore used for configurations with positive length on every bond. Every zero-energy chain belongs to this domain, since its bond lengths equal \(\bl\).
We now \emph{extend} the range of the integers $k$ and consider a discrete system on the \emph{real line}. 

\begin{definition}
The stationary points of the time-integrated action associated with \eqref{eq:action} satisfy at interior sites of a finite chain the following {\bf discrete model of rod linkage}, which we also adopt as the local equation for the infinite chain:
\bel{eq:EL}
\ddot \theta_k + \Big(
\del_2 \Gbf(\theta_{k-1}, \theta_k) + \del_1 \Gbf(\theta_k, \theta_{k+1}) 
\Big) = 0, 
\ee
in which the angle variables $\theta_k=\theta_k(t)$ ($k=\ldots, -1, 0, 1, \ldots$) are defined for all integers $k$ and all times $t \geq 0$. The dependence on the mesh scale is encoded in the scaling convention for $\Gbf$ and must be fixed together with the time normalization.
\end{definition}

Equation~\eqref{eq:EL} also applies to finite or periodic chains with the corresponding boundary conditions; Section~\ref{sec:numerics} uses periodic conditions. For fixed dimensionless geometry, time rescaling acts directly on the rescaled stiffness: if \(\{\theta_k(t)\}\) solves \eqref{eq:EL} with coefficient \(\Kt\), then \(\{\theta_k(qt)\}\) solves it with coefficient \(q^2\Kt\).

The model above is immediately generalized by considering a general \textbf{pair-interaction Lagrangian}
\bel{eq:action-g}
\mathcal L_N = \sum_{k=0}^N  \frac{1}{2} \dot \theta_k^2 -   \sum_{k=0}^{N-1} \Gbf(\theta_k, \theta_{k+1}). 
\ee 
This framework avoids writing the cumbersome expressions for $\Gbf$ whenever this is not necessary in our discussion. The precise assumptions which should be made on $\Gbf$ to replicate the behavior of the physical model, or to exhibit simpler behavior of the same type, are explored in the following section.


\section{Abstract zero-energy chains and continuum consistency}
\label{sec:3}

\subsection{Graph-type zero sets: the one-periodic case} 
\label{sec:1periodic}

We consider an infinite nearest-neighbor chain with angles denoted by 
\(\theta_k\in S^1\) and nonnegative pair potential
\(\Gbf_\s(\theta_k,\theta_{k+1})\), where \(\s\to0\) is the continuum parameter. We have two distinct goals here: first, to see how the zero set of \(\Gbf_\s\) defines a discrete formula for stationary zero-energy chains, and how under graph or selected-branch hypotheses this discrete formula has an ODE limit. Second, check that sampling smooth time-dependent fields in the lattice equation \emph{formally} gives limiting continuum residual expansions.  

Set \(M=S^1\times S^1\), and let \(\pi_1,\pi_2\) denote its coordinate projections. We identify \(S^1\) with \([0,2\pi)\); when \(\b<\a\), we use the \emph{cyclic interval} notation \((\a,\b)\) to mean \((\a,2\pi)\cup[0,\b)\). For any finite cyclically ordered fixed-point set \(\{\a_i\}_{i=1}^I\) below, we choose lifts \(\a_1<\cdots<\a_I<\a_1+2\pi\) and set \(\a_{I+1}=\a_1+2\pi\). If \(I=1\), the endpoints of this lifted interval represent the same circle point. Consider a smooth family
\be
\Gbf:[0,1]\times M\longrightarrow[0, + \infty),
\qquad (\s,z)\longmapsto\Gbf_\s(z).
\ee
Let \(Z_\s:=\{z\in M:\Gbf_\s(z)=0\}\) be the zero-energy set and let \(D:=\{(z,z)\in M:z\in S^1\}\) be the diagonal of the state space  \(M=S^1\times S^1\). Define a map from the circle to itself, say \(g_\s:S^1\to S^1\), by requiring
\((\theta,g_\s(\theta))\in Z_\s\). We assume the following condition on the Lagrangian $\Gbf$ and on the map $g_\s$. 
\begin{enumerate}
	\item[{ (H1)}] The zero-energy set $Z_\s$ is a smooth \emph{closed curve} in $M$.

	\vskip.15cm
	
	\item[{ (H2)}] The coordinate projection $\pi_i|_{Z_\s}$ is a diffeomorphism for each $i$, and the induced map \(g_\s\) is orientation preserving.
	\vskip.15cm
	
		\item[{ (H3)}] For \(0<\s\leq1\), the intersection
		\(Z_\s\cap D=\{(\a_i,\a_i):1\leq i\leq I\}\) is \emph{nonempty, finite}, and \emph{independent} of \(\s\). At \(\s=0\), the zero-energy set is  \(Z_0=D\).
	\vskip.15cm
	
	\item[{ (H4)}] One has $\operatorname{dist}^2(z,Z_\s) \lesssim \Gbf_\s(z) \lesssim \operatorname{dist}^2(z,Z_\s)$, uniformly in $\s$. 
	\vskip.15cm
	
		\item[{ (H5)}] For the natural lift of the map $g_\s$ to the real line, say \(\widetilde g_\s:\R\to\R\), one has the global bound 
		\(\|\widetilde g_\s-\operatorname{id}_{\R}\|_{C^1} \leq C\s\).
	\vskip.15cm
	
		\item[{ (H6)}] For every fixed \(x\in D\setminus\{(\a_i,\a_i):1\leq i\leq I\}\), there is \(c_x>0\), independent of \(\s\), such that \(d(x,Z_\s)\geq c_x\s\) for all \(0<\s\leq1\).
\end{enumerate}
In particular, \textup{(H2)} guarantees that \(g_\s\) is well-defined. In \textup{(H6)}, a point \(x\in D\) is identified with its image in \(M\). Moreover, 	all differences involving \(g_\s\), and all piecewise-linear angle interpolants below, are taken in the lift specified in \textup{(H5)}. 

The corresponding \textbf{infinite-chain Lagrangian} is defined as 
\be
\mathcal L_\s(\theta,\dot\theta)
:=\sum_{k\in\Z}\left(
\frac12\dot\theta_k^2-\Gbf_\s(\theta_k,\theta_{k+1})
\right);
\ee
however, since this sum need not converge, we use the associated local equation as the definition: a collection of \(C^2\) curves
\(\theta_k:[0,T]\to S^1\) (with \(k\in\Z\)) is called an \textbf{infinite-chain \(\s\)-solution} if
\bel{eq:simplediscretepde}
\ddot{\theta}_k+\p_2\Gbf_\s(\theta_{k-1},\theta_k)
+\p_1\Gbf_\s(\theta_k,\theta_{k+1})=0
\qquad(k\in\Z).
\ee
For finite or periodic chains, and for compactly supported variations of an infinite chain, this is the Euler--Lagrange equation of the corresponding time-integrated action.

A zero-energy chain is a time-independent configuration for which
\(\Gbf_\s(\theta_k,\theta_{k+1})=0\)
for every \(k\). Since \(\Gbf_\s\geq0\) is smooth, its gradient vanishes on its zero set, so every such chain is an equilibrium solution of \eqref{eq:simplediscretepde}. The classification below concerns stationary chains; it does not classify time-dependent states with vanishing bond potential. Other equilibria may have positive energy on some bonds.

Under the above assumptions, we may classify the zero-energy chains. 

\begin{lemma}\label{lem:one-periodic-classification}
Let $\Gbf_\s$ satisfy \textup{(H1)--(H6)} above, with \(\s>0\) sufficiently small. Let $\{\a_i\}_{i=1}^I$ be the values for which $(\a_i,\a_i)\in Z_\s \cap D$, ordered cyclically. Then \emph{for every angle} $\g \in S^1$, there is a \emph{unique zero-energy chain} $\{\theta_k\}$ with $\theta_0 = \g$. 
\bei 

\item If $\g = \a_i$, this chain is constant.

\item If $\g \in (\a_i,\a_{i+1})$ the chain is strictly monotone in a lift of this interval, with
	\be
	\lim_{k\rightarrow \pm \8} \theta_k \in \{\a_i,\a_{i+1}\}.	
	\ee
	\eei
	\noindent 
Moreover, the chains satisfy
\bel{eq:1}
		\operatorname{dist}_{S^1}(\theta_k,\theta_{k+1})\leq C \s.
\ee
\end{lemma}

\begin{proof}
	By \textup{(H2)}, $g_\s$ is a diffeomorphism. By definition, the $\a_i$ are precisely its fixed points, and $g_\s$ maps each interval $(\a_i,\a_{i+1})$ monotonically onto itself. A sequence $\{\theta_k\}$ is a zero-energy chain if and only if $\theta_{k+1} = g_\s(\theta_k)$ for every $k$; iteration in both directions gives existence and uniqueness. If the orbit is not fixed, its lift is strictly monotone and its limits are fixed points, hence the two endpoints of the interval. Estimate \eqref{eq:1} follows from \textup{(H5)}.
\end{proof}

\begin{theorem}[Continuum limit of graph-type zero-energy chains]
\label{thm:one-periodic-limit}
Let $\Gbf_\s$ satisfy \textup{(H1)--(H6)} above and let \(\s_j\to0\), and suppose that the following limit exists: 
\be
f_{\s_j}:=\frac{\widetilde g_{\s_j}-\operatorname{id}_{\R}}{\s_j}
\longrightarrow f
\quad\hbox{uniformly on }S^1.
\ee
For \(\g\in S^1\), let \(\{\theta_k^{(j)}\}_{k\in\Z}\) be the zero-energy chain from Lemma~\ref{lem:one-periodic-classification}, and let \(\theta_{\s_j}\) be its piecewise-linear interpolation defined by
\(\theta_{\s_j}(k\s_j)=\theta_k^{(j)}\), using the consistent lift fixed above. Then one has 
\be
\theta_{\s_j}\longrightarrow\theta
\quad\hbox{locally uniformly on }\R,
\ee
where \(\theta\) is the unique solution of
\bel{eq:one-periodic-ode}
\theta'=f(\theta),\qquad \theta(0)=\g.
\ee
If \(\g\in(\a_i,\a_{i+1})\), the trajectory remains in that interval and connects its endpoint equilibria in the direction determined by the sign of \(f\).
\end{theorem}

\begin{proof}
\bse
Estimate \eqref{eq:1} gives a uniform Lipschitz bound for the interpolants, so every sequence has a locally uniformly convergent subsequence. For \(x<y\), after choosing a lift on the relevant arc,
\be
\theta_{\s_j}(y)-\theta_{\s_j}(x)
=\int_x^y f_{\s_j}\big(\theta_{\s_j}(s)\big)\,ds+O(\s_j),
\ee
where the error is uniform on compact intervals. Passing to the limit gives the integral form of \eqref{eq:one-periodic-ode}. Hypothesis \textup{(H5)} bounds \(f_{\s_j}\) in \(C^1\); hence \(f\) is Lipschitz and the initial-value problem is unique. Every convergent subsequence therefore has the same limit, which proves convergence of the full sequence. For a fixed interior point \(x\), the graph property and the uniform Lipschitz bound in \textup{(H5)} make
\(\operatorname{dist}((x,x),Z_\s)\) comparable to
\(\operatorname{dist}_{S^1}(x,g_\s(x))\).

\needspace{6\baselineskip}\noindent Indeed, with \(L:=\sup_{0<\s\leq1}\operatorname{Lip}(g_\s)<\infty\), the triangle inequality gives
\be
\operatorname{dist}_{S^1}(x,g_\s(x))\leq\sqrt{1+L^2}\,\operatorname{dist}((x,x),Z_\s).
\ee
The reverse bound follows by choosing \((x,g_\s(x))\in Z_\s\). Thus \textup{(H6)} and uniform convergence imply \(f(x)\neq0\) away from the \(\a_i\). Since \(g_\s(\a_i)=\a_i\), one also has \(f(\a_i)=0\). Its sign is constant on each intervening arc, and the final assertion follows.
\ese
\end{proof}


\subsubsection{Leading continuum residual}

To identify local residuals for smooth time-dependent fields as $\s \rightarrow 0$, we note that from the smoothness of $\Gbf_\s$ in the $\s$ parameter, we have 
\be
\|\Gbf_\s - \Gbf_0\|_{C^k} \rightarrow 0
\ee
for any $k$. In this case, we have 
\be \label{eq:Gconv}
|\Gbf_\s(\theta, \theta)| \leq C \s^2, \qquad 
|\n \Gbf_\s(\theta, \theta)| \leq C \s,
\ee
Since \(\Gbf_0\) vanishes on the diagonal and is comparable to its squared distance, its Hessian vanishes in the tangent direction and is positive in the normal direction. Hence
\be \label{eq:ddGconv}
|D^2\Gbf_\s(\theta, \theta) - c_0(\theta) \left(\begin{matrix} 1 & -1 \\ -1 & 1 \end{matrix} \right) | \leq C \s,
\ee
and
\be \label{eq:dddGconv}
|D^k \Gbf_\s(\theta, \theta) - D^k \Gbf_0(\theta,\theta)| \leq C_k \s, \qquad {k>2}.
\ee
Here \(c_0(\theta)>0\), and \(C_k\) may depend on \(k\). Indeed, \textup{(H4)--(H5)} yield \(h_\s(\theta):=\Gbf_\s(\theta,\theta)=O(\s^2)\). Joint smoothness then gives \(\del_\s h_0=0\) and \(\del_\theta h_\s=O(\s^2)\), uniformly in \(\theta\). Thus
\be \label{eq:dGconvimproved}
|\p_1 \Gbf_\s(\theta, \theta) + \p_2 \Gbf_\s(\theta, \theta)| \leq C \s^2.
\ee

We write $\s=\b a$ with fixed $\b\geq0$ and fix a time-rescaling factor $q$. We sample a smooth field by $\theta_k(t)=\theta(ka,qt)$ and compare the resulting discrete residual with candidate continuum equations as $a\to0$. Here $a$ is the lattice spacing in the limiting spatial variable. Writing the equation in the variable $T=qt$ divides the effective stiffness by $q^2$; conversely, replacing a solution $\theta(t)$ by $\theta(qt)$ at fixed time multiplies its stiffness by $q^2$, as in Section~\ref{sec:2}. We keep $q$ outside the definition of $\Gbf_\s$ so that the assumptions on the potential remain explicit.
In the formal continuum displays below, dots on \(\theta(x,T)\) denote derivatives with respect to \(T=qt\); dots on the discrete angles \(\theta_k(t)\) denote derivatives with respect to the physical time \(t\).

Recall the discrete equation \eqref{eq:simplediscretepde} satisfied by $\theta_k$, and insert $\theta_k = \theta(x, t q)$, $\theta_{k\pm 1} =  \theta(x, t q) \pm a \theta'(x, t q) + \frac{a^2}{2} \theta''(x,tq)+O(a^3)$ (here $x = k a$):
\be
\begin{aligned}
-q^2 \ddot{\theta}  &=  \p_2 \Gbf_\s (\theta - a \theta' + \frac{a^2}{2} \theta'', \theta) 
+ \p_1 \Gbf_\s (\theta, \theta + a \theta' + \frac{a^2}{2} \theta'') +O(a^3) \\
   &\approx  [\p_{1}\Gbf_\s(\theta,\theta) + \p_{2}\Gbf_\s(\theta,\theta)] + \p_{12} \Gbf_\s(\theta,\theta)a^2 \theta''  + [\p_{112}\Gbf_\s(\theta,\theta) + \p_{122}\Gbf_\s(\theta,\theta)] \frac{a^2}{2}(\theta')^2\\
   &\qquad  + O(a^3).
\end{aligned}
\ee
Here \(\theta\) and its derivatives are evaluated at \((x,tq)\). 

The terms in this expansion have three distinct orders.
\bei
\item The zero-order force is $O(a^2)$ by \eqref{eq:dGconvimproved}.
\item The second-derivative term is $O(a^2)$ because $D^2\Gbf_\s(\theta,\theta)=O(1)$ by \eqref{eq:ddGconv}.
\item The gradient-square term is generally $O(a^2)$, but it is $O(a^3)$ when $D^3\Gbf_0(\theta,\theta)=0$ for every $\theta$.
\eei
We choose $q=a$ and divide by $a^2$ to obtain a PDE with terms of order one:
\be
0 = \ddot{\theta} + \frac{1}{a^2}[\p_{1}\Gbf_\s(\theta,\theta) + \p_{2}\Gbf_\s(\theta,\theta)] + \p_{12} \Gbf_\s(\theta,\theta) \theta''  + [\p_{112}\Gbf_\s(\theta,\theta) + \p_{122}\Gbf_\s(\theta,\theta)] \frac{1}{2}(\theta')^2 + O(a).
\ee
Note that the term with second derivatives of \(\theta\) has a negative sign (from \eqref{eq:ddGconv}), making this a hyperbolic equation. The first term is of zero order (so a semilinear Klein--Gordon type effect). When $D^3\Gbf_0(\theta,\theta)=0$ for every $\theta$, the last term is of the same order as the error in this approximation. Note, furthermore, that from \eqref{eq:ddGconv} replacing $\p_{12} \Gbf_\s(\theta,\theta)$ by $\p_{12} \Gbf_0(\theta,\theta)$ incurs errors of order at most $a$ as well. The same cannot be done for the first term (it would create an order-one error). Therefore, a simplified equation at the same level of approximation, under $D^3\Gbf_0(\theta,\theta)=0$, would be
\be \label{eq:PDEsimple}
0 = \ddot{\theta} + \frac{1}{a^2}[\p_{1}\Gbf_\s(\theta,\theta) + \p_{2}\Gbf_\s(\theta,\theta)] + \p_{12} \Gbf_0(\theta,\theta) \theta'' + O(a).
\ee


Let us compute this derivation for our specific example, given by 
\bel{eq:simple-potential}
\Gbf_\s(\theta, \omega) = \sin^2\!\left(\frac{\omega - \theta - \sigma \cos \theta}{2}\right). 
\ee
For the stationary-chain hypotheses, this example is restricted to \(\sigma<1\): at \(\sigma=1\), the derivative of \(g_\sigma(\theta)=\theta+\sigma\cos\theta\) vanishes at \(\theta=\pi/2\), so \textup{(H2)} fails. The consistency calculations below use only \(\sigma=\beta a\to0\).
For a literal application of Theorem~\ref{thm:one-periodic-limit} on its full parameter interval, fix $0<\s_*<1$ and use $\widehat{\Gbf}_\eta:=\Gbf_{\s_*\eta}$ for $0\leq\eta\leq1$. The hypotheses follow from two observations.
\bei
\item The zero set $\widehat Z_\eta:=\{\widehat{\Gbf}_\eta=0\}$ is the graph of $\widehat g_\eta(\theta)=\theta+\s_*\eta\cos\theta$, with $\widehat g_\eta'(\theta)\geq1-\s_*>0$. For $\eta>0$ its fixed points are $\pi/2$ and $3\pi/2$. The displayed formula defines a real lift that is $O(\eta)$-close to the identity in $C^1$, so hypotheses \textup{(H1)--(H3)} and \textup{(H5)} hold.
\item Let $L:=1+\s_*$ and $\delta:=\operatorname{dist}_{S^1}(y,\widehat g_\eta(x))$. The product metric gives $\delta/\sqrt{1+L^2}\leq\operatorname{dist}((x,y),\widehat Z_\eta)\leq\delta$. Since $\widehat{\Gbf}_\eta(x,y)=\sin^2(\delta/2)$, this proves \textup{(H4)}; at $y=x$, one has $\delta=\s_*\eta|\cos x|$, which proves \textup{(H6)} away from the two fixed points.
\eei
Its derivatives satisfy
\be
\aligned
\del_1 \Gbf_\s(\theta, \omega) 
& = \sin\!\left(\frac{\omega - \theta - \sigma \cos \theta}{2}\right)
\cos\!\left(\frac{\omega - \theta - \sigma \cos \theta}{2}\right)
\big( -1 + \sigma \sin \theta  \big),
\\
\del_2 \Gbf_\s(\theta, \omega) 
& = \sin\!\left(\frac{\omega - \theta - \sigma \cos \theta}{2}\right)
\cos\!\left(\frac{\omega - \theta - \sigma \cos \theta}{2}\right),
\\
\del_{12} \Gbf_0(\theta,\theta) 
& = \frac{-1}{2}, 
\hskip2.cm 
\del_{112} \Gbf_0(\theta,\theta) = \del_{122} \Gbf_0(\theta,\theta) = 0.
\endaligned
\ee
The approximate PDE \eqref{eq:PDEsimple} then becomes (using $\s = \b a$)
\be
0 = \ddot \theta + \frac{\b}{2a} \Big( - \sin (\b a \cos \theta) \sin \theta \Big) - \frac{1}{2} \theta''.
\ee
Up to omitting a term of order at most $a$, the second term may be further simplified and we arrive at the following equation. 

The leading continuum model is
\be
\ddot \theta  - \frac{1}{2} \theta'' - \frac{\b^2}{4} \sin 2 \theta = 0, 
\ee

This is a \emph{nonlinear Klein--Gordon equation} with wave speed $1/\sqrt{2}$. We observe that in the range $[\pi/2, 3\pi/2]$ of interest  the term $\sin( 2 \theta)$ vanishes at the end points $\pi/2$ and $3\pi/2$ as well as in the middle of this interval $\pi$, a behavior typical of a \emph{two-well model of phase transition.} 

\begin{lemma}
\label{lem:bond-phase-expansion}
Let \(x=ka\), \(T=at\), and write \(u=\theta(x,T)\). Define
\be
z_k^a:=\theta_{k+1}^a-\theta_k^a-\b a\cos\theta_k^a.
\ee
If \(\theta\in C^4\), then, with all derivatives evaluated at \((x,T)\), one has 
\be
\begin{aligned}
z_k^a={}&a(u_x-\b\cos u)+\frac{a^2}{2}u_{xx}
+\frac{a^3}{6}u_{xxx}+O(a^4),\\
z_{k-1}^a={}&a(u_x-\b\cos u)
-a^2\left(\frac12u_{xx}+\b u_x\sin u\right)\\
&+a^3\left(\frac16u_{xxx}
+\frac{\b}{2}u_{xx}\sin u
+\frac{\b}{2}u_x^2\cos u\right)+O(a^4).
\end{aligned}
\ee
Consequently, the discrete force for \eqref{eq:simple-potential} satisfies
\bel{eq:bond-force-expansion}
\begin{aligned}
&\del_2\Gbf_{\b a}(\theta_{k-1}^a,\theta_k^a)
+\del_1\Gbf_{\b a}(\theta_k^a,\theta_{k+1}^a)\\
& =-a^2\left(\frac12u_{xx}
+\frac{\b^2}{4}\sin(2u)\right)
+a^3\left(\frac{\b}{2}\sin u\,u_{xx}
+\frac{\b}{4}\cos u\,u_x^2\right)+O(a^4).
\end{aligned}
\ee
The remainders are uniform when the indicated derivatives range over a compact set.
\end{lemma}

\begin{proof}
Taylor expansion at \(x\) gives the two displayed formulas for \(z_k^a\) and \(z_{k-1}^a\). Moreover,
\be
\begin{aligned}
\del_2\Gbf_{\b a}(\theta_{k-1}^a,\theta_k^a)
&=\frac12\sin z_{k-1}^a,\\
\del_1\Gbf_{\b a}(\theta_k^a,\theta_{k+1}^a)
&=\frac12\big(-1+\b a\sin u\big)\sin z_k^a.
\end{aligned}
\ee
Substituting the phase expansions and using
\(\sin z=z-z^3/6+O(z^5)\) gives \eqref{eq:bond-force-expansion}. The cubic contributions have no term of order \(a^3\) because \(z_{k-1}^a-z_k^a=O(a^2)\) and both phases are \(O(a)\); the cubic term multiplied by \(\b a\) is also \(O(a^4)\).
\end{proof}

\needspace{12\baselineskip}
\begin{proposition}[Continuum limits: leading-order consistency]
\label{prop:first-order-consistency}
Let \(\theta\) be \(C^4\) with bounded derivatives on a compact space--time region, set \(\s=\b a\), and sample
\be
\theta_k^a(t):=\theta(ka,at).
\ee
For the potential in \eqref{eq:simple-potential}, define the discrete residual
\be
\mathcal R_k^a:=\ddot\theta_k^a+
\del_2\Gbf_{\b a}(\theta_{k-1}^a,\theta_k^a)
+\del_1\Gbf_{\b a}(\theta_k^a,\theta_{k+1}^a).
\ee
Then, uniformly on compact sets,
\be
a^{-2}\mathcal R_k^a
=\del_T^2\theta-\frac12\theta_{xx}
-\frac{\b^2}{4}\sin(2\theta)+O(a).
\ee
Thus a smooth solution of the leading Klein--Gordon equation has discrete residual \(O(a^3)\) under this sampling.
\end{proposition}

\begin{proof}
Since \(\ddot\theta_k^a=a^2\del_T^2\theta(ka,at)\), adding the inertial term to \eqref{eq:bond-force-expansion} and dividing by \(a^2\) gives the assertion.
\end{proof}


\subsubsection{First correction}

We now retain the \emph{next power} of \(a\). Using the expansion 
\be
\theta_{k\pm 1} =  \theta(x, t q) \pm a \theta'(x, t q) + \frac{a^2}{2} {\theta''(x,tq)} \pm \frac{a^3}{6} {\theta'''(x,tq)}+O(a^4),
\ee
we have 
\be
\begin{aligned}
-q^2 \ddot \theta  &\approx  [\p_{1}\Gbf_\s(\theta,\theta) + \p_{2}\Gbf_\s(\theta,\theta)] + \p_{12} \Gbf_\s(\theta,\theta)a^2 \theta''  + [\p_{112}\Gbf_\s(\theta,\theta) + \p_{122}\Gbf_\s(\theta,\theta)] \frac{a^2}{2}(\theta')^2\\
&\qquad { +[-\p_{112}\Gbf_\s(\theta,\theta)+\p_{122}\Gbf_\s(\theta,\theta)]
\frac{a^3}{2}\theta'\theta''}\\
&\qquad	 + [- \p_{1112}\Gbf_\s(\theta,\theta) + \p_{1222}\Gbf_\s(\theta,\theta)] \frac{a^3}{6} (\theta')^3 
+ O(a^4).
\end{aligned}
\ee
Note that there are no terms featuring third derivatives: they cancel. If \(D^3\Gbf_0(\theta,\theta)=0\) for every \(\theta\), which we assume in the computations below, the newly displayed mixed term is \(O(a^4)\). The first two terms on the right are \(O(a^2)\), while the remaining displayed contributions are \(O(a^3)\) or smaller in magnitude. Replacing $\Gbf_\s$ with $\Gbf_0$ in the fourth-derivative term contributes error of order $a^4$ at most; setting $q = a$ as before, we then obtain the formal normalized expansion
\be
\begin{aligned}
0  & = \ddot \theta  + \frac{1}{a^2}[\p_{1}\Gbf_\s(\theta,\theta) + \p_{2}\Gbf_\s(\theta,\theta)] + \p_{12} \Gbf_\s(\theta,\theta) \theta''  + [\p_{112}\Gbf_\s(\theta,\theta) + \p_{122}\Gbf_\s(\theta,\theta)] \frac{1}{2}(\theta')^2\\
&\qquad	 + [- \p_{1112}\Gbf_0(\theta,\theta) + \p_{1222}\Gbf_0(\theta,\theta)] \frac{a}{6} (\theta')^3 + { O(a^2)}.
\end{aligned}
\ee

For the previous example, this yields
\be
\begin{aligned}
&0 = \ddot \theta + \frac{\b}{2a} \Big( - \sin (\b a \cos \theta) \sin \theta \Big) + \frac{1}{2} \cos(\b a \cos(\theta))(-1 + \b a \sin \theta) \theta'' \\
&+ \frac{(\theta')^2}{4}[\sin(\b a \cos(\theta))(-1 + \b a\sin \theta)^2 + \cos(\b a \cos(\theta))(\b a\cos \theta) + \sin(\b a \cos(\theta))(-1 + \b a\sin \theta)  ].
\end{aligned}
\ee
Note the absence of the $(\theta')^3$ term: $[- \p_{1112}\Gbf_0(\theta,\theta) + \p_{1222}\Gbf_0(\theta,\theta)] = 0$ for this example. We simplify further, removing terms of order higher than $a$:
\be
0 = \ddot \theta - \frac{\b^2}{4}\sin 2 \theta + \frac{1}{2} (-1 + \b a \sin \theta) \theta'' + \frac{(\theta')^2}{4} \b a \cos(\theta).
\ee
The first corrected continuum model is
\be
\ddot\theta-\frac12\theta''-\frac{\b^2}{4}\sin(2\theta)
=-\frac{\b a}{4}\Big(2\sin\theta\,\theta''+\cos\theta\,(\theta')^2\Big).
\ee

\needspace{10\baselineskip}
\begin{proposition}[Continuum limits: first corrected continuum equation]\label{prop:second-order-consistency}
Under the assumptions and sampling of Proposition~\ref{prop:first-order-consistency}, with \(\theta\in C^5\), the normalized residual satisfies
\be
\begin{aligned}
a^{-2}\mathcal R_k^a
={}&\del_T^2\theta-\frac12\theta_{xx}-\frac{\b^2}{4}\sin(2\theta)
+\frac{\b a}{4}\big(2\sin\theta\,\theta_{xx}
+\cos\theta\,(\theta_x)^2\big)+O(a^2)
\end{aligned}
\ee
uniformly on compact sets. In particular, for a family of solutions of the corrected continuum equation with uniform \(C^5\) bounds on the compact set, the discrete residual is \(O(a^4)\), one power of \(a\) smaller than for the leading equation.
\end{proposition}

\begin{proof}
The equation \eqref{eq:bond-force-expansion}, together with
\(\ddot\theta_k^a=a^2\del_T^2\theta(ka,at)\), contains every term through order \(a^3\). Division by \(a^2\) gives the stated corrected residual. In particular, the apparently possible cubic Taylor terms do not contribute at this order, as shown in the last sentence of Lemma~\ref{lem:bond-phase-expansion}.
\end{proof}

Whereas Theorem~\ref{thm:one-periodic-limit} is a convergence result for zero-energy chains, Propositions~\ref{prop:first-order-consistency} and \ref{prop:second-order-consistency} only concern the local residual of smooth time-dependent fields; they do not imply dynamical convergence.


\subsection{Two graph components: selected two-periodic chains} 
\label{sec:2periodic}

The key feature of the 1-periodic case was the simple structure of the underlying discrete dynamical system for zero-energy solutions, and the associated set $Z_\s$. We now turn our attention to a more complicated structure for $Z_\s$, which is essentially a perturbation of a \emph{pair} of curves in $M$. 

Indeed, we make the following alternative assumptions about $\Gbf_\s$:
\begin{enumerate}
	\item[{ (P1)}] $Z_\s$ is the disjoint union of two smooth closed curves $Z_\s^1,Z_\s^2$ in $M$.
	\vskip.15cm
	
	\item[{ (P2)}] $\pi_i|_{Z_\s^j}$ is a diffeomorphism for each $i,j$.
		\vskip.15cm
		
	\item[{ (P3)}] For \(0<\s\leq1\), the set of fixed points of
	\(g^1_\s\circ g_\s^2\), denoted by \(\{\a_i\}_{i=1}^I\), is nonempty, finite, and independent of \(\s\).
		\vskip.15cm
		
		\item[{ (P4)}] $\operatorname{dist}^2(z,Z_\s) \lesssim \Gbf_\s(z)\lesssim \operatorname{dist}^2(z,Z_\s)$, uniformly in $\s$.
		\vskip.15cm
		
		\item[{ (P5)}] There are real representatives $\b_1,\b_2$ with $\b_1+\b_2=0$ and compatible lifts $\widetilde g_\s^i:\R\to\R$ of $g_\s^i$ such that
	$\|\widetilde g_\s^i-\operatorname{id}_{\R}-\b_i\|_{C^1([0,2\pi])}\leq C\s$ for $i=1,2$.
	\vskip.15cm
	
		\item[{ (P6)}] For every $x\in\R$ with $x\not\equiv\a_i\pmod{2\pi}$, there is $c_x>0$ independent of $\s$ such that $|\widetilde g_\s^1\circ\widetilde g_\s^2(x)-x|\geq c_x\s$.
		\vskip.15cm
		
		\item[{ (P7)}] \(Z_\s\cap D=\emptyset\) for \(0<\s\leq1\).
\end{enumerate}
Here \(g_\s^j:S^1\to S^1\) is defined by \((\theta,g_\s^j(\theta))\in Z_\s^j\). Assumption \textup{(P5)} gives \(\|\widetilde g_\s^2\circ\widetilde g_\s^1-\operatorname{id}_{\R}\|_{C^1([0,2\pi])}\leq C\s\). Assumption \textup{(P7)} excludes diagonal bonds for each positive \(\s\), but it does not by itself imply \(\b_i\neq0\) in the limit.

\begin{proposition}[Selected two-periodic chains]
\label{prop:selected-two-periodic}
Assume \textup{(P1)--(P7)} and \(0<\s\ll1\). For every \(\g\in S^1\), there is a unique zero-energy chain in the alternating component class
\be
(\theta_{2k},\theta_{2k+1})\in Z_\s^2,\qquad
(\theta_{2k+1},\theta_{2k+2})\in Z_\s^1,
\ee
with \(\theta_0=\g\). Its parity subsequences satisfy
\be
\label{eq:2}
\operatorname{dist}_{S^1}(\theta_k,\theta_{k+2})\leq C\s.
\ee

If \(\g=\a_i\), this chain is period two. If
\(\g\in(\a_i,\a_{i+1})\), its even-site lift is strictly monotone and converges to the fixed endpoints of the lifted interval as \(k\to\pm\infty\). When \(I=1\), these endpoints represent the same point of \(S^1\).
Suppose that, along \(\s_j\to0\), the following $2\pi$-periodic functions converge uniformly on $S^1$:
\be
F_{\s_j}(x):=\frac{\widetilde g_{\s_j}^1(\widetilde g_{\s_j}^2(x))-x}{2\s_j}
\longrightarrow F(x),\qquad x\in\R.
\ee
Then, using the compatible real lifts in \textup{(P5)}, the piecewise-linear interpolants defined by
\(\theta_{\s_j}(2k\s_j)=\theta_{2k}\) converge locally uniformly to the solution of
\be
\theta'=F(\theta),\qquad \theta(0)=\g.
\ee
For interior initial data, this limiting trajectory connects the \emph{same endpoint equilibria} in the compatible real lift.
The translated component class gives the companion family.
\end{proposition}

\begin{proof}
The recurrence is
\(\theta_{2k+2}=g_\s^1\circ g_\s^2(\theta_{2k})\), with odd terms then determined by \(g_\s^2\). The maps are diffeomorphisms, so iteration in both directions gives existence and uniqueness. Estimate \eqref{eq:2} follows from \textup{(P5)}. For small \(\s\), the two-step map is orientation preserving by \textup{(P5)}. Assumptions \textup{(P3)} and \textup{(P6)} therefore give the fixed and monotone alternatives on every intervening arc.

The continuum limit is the same compactness and Riemann-sum argument as in Theorem~\ref{thm:one-periodic-limit}, now with spatial increment \(2\s\). Assumption \textup{(P5)} makes the normalized maps \(F_\s\) uniformly Lipschitz. It also makes the compatible two-step lift \(O(\s)\)-close to the identity, so the circle fixed points in \textup{(P3)} satisfy \(F_\s(\a_i)=0\) for small \(\s\), and hence \(F(\a_i)=0\). After normalized convergence, \textup{(P6)} implies that \(F\) has no interior zero, which proves the endpoint assertion for the limiting ODE.
\end{proof}


Shifting a chain in the preceding class by one lattice site gives the companion family. In the first class, the bonds in $Z_\s^2$ occur at even indices:
\be
(\theta_k,\theta_{k+1}) \in Z_\s^2 \quad \text{ if and only if } k \text{ is even},
\ee
in the shifted class, they occur at odd indices:
\be
(\theta_k,\theta_{k+1}) \in Z_\s^2 \quad \text{ if and only if } k \text{ is odd}.
\ee

More generally, a zero-energy chain is determined by a sequence $\{\t_k\}_{k\in\Z}$ with $\t_k\in\{1,2\}$ and an initial value $\g\in S^1$. The recurrence is
\be
\theta_k = g^{\t_{k-1}}_\s\circ \ldots \circ g^{\t_j}_\s(\theta_j)
\ee
for any $j<k$, with $\theta_0=\g$. Without additional restrictions on the component sequence, these chains need not have an ODE limit as $\s\to0$.

For the two graph components, adjacent angles need not be close even though the equal-parity increments are \(O(\s)\) for the alternating component classes by \eqref{eq:2}. A time-dependent comparison therefore uses separate fields for the two parity classes. Lemma~\ref{lem:even-odd-system} gives the exact two-cell equations for the physical linkage, and Proposition~\ref{prop:two-component-residual} gives their Taylor residual with spatial step \(h=2a\).

The corresponding stability problem remains open.
\begin{question}[Dynamical persistence of the two-cell manifold]
\label{question:persistence}
Let \(\s=\b a\) with fixed \(\b>0\), and let \(g_\s^2\) denote the selected even-to-odd graph from \textup{(P1)--(P7)}. For a periodic chain, or a finite chain with specified boundary conditions, define its distance from this graph in rescaled phase space by
\be
\label{eq:branch-phase-space-distance}
\begin{aligned}
E_\s(t):=\sup_k\bigg[&
\operatorname{dist}_{S^1}\big(\theta_{2k+1}(t),g_\s^2(\theta_{2k}(t))\big)\\
&+a^{-1}\big|\dot\theta_{2k+1}(t)-(g_\s^2)'(\theta_{2k}(t))\dot\theta_{2k}(t)\big|\bigg].
\end{aligned}
\ee
Which material scaling, energy bounds, and regularity assumptions yield constants \(T,C,c>0\), independent of \(a\), for which
\be
\label{eq:branch-persistence-question}
E_\s(0)\leq c\s
\quad\Longrightarrow\quad
\sup_{|t|\leq T/a}E_\s(t)\leq C\s?
\ee
\end{question}
The distance in \eqref{eq:branch-phase-space-distance} uses the exact selected graph and includes the transverse velocity; proximity of the angles alone cannot control later motion. The question asks whether this branch remains close on a fixed rescaled time interval. The numerical tests in Section~\ref{sec:numerics} concern a smooth one-component periodic comparison and do not address this estimate.
\subsection{Curve-valued recurrence in the non-monotone case}

The graph structure used above \emph{fails in the flipper regime} because the coordinate projections have \emph{folds}. The recurrence can nevertheless be formulated on selected arcs of the constraint curve itself. We therefore impose the following intrinsic assumptions.
\begin{enumerate}
	\item[{ (N1)}] $Z_\s$ is a smooth closed curve in $M$.
	\vskip.15cm
	
	\item[{ (N2)}] There are a \emph{pair of closed arcs} $Z_\s^1,Z_\s^2$ in $Z_\s$ and a pair of diffeomorphisms $g_\s^1: Z_\s^1 \rightarrow Z_\s^2$ and $g_\s^2: Z_\s^2 \rightarrow Z_\s^1$, with the property that $\pi_1 \circ g^i_\s = \pi_2|_{Z^i_\s}$.	\vskip.15cm
	
		\item[{ (N3)}] For \(0<\s\leq1\), the set of fixed points of
		\(g^1_\s\circ g_\s^2\), denoted by \(\{\a_i\}_{i=1}^2\), consists precisely of the two endpoints of \(Z_\s^2\) and is independent of \(\s\).	\vskip.15cm
	
	\item[{ (N4)}] $ \operatorname{dist}^2(z,Z_\s) \lesssim \Gbf_\s(z)\lesssim \operatorname{dist}^2(z,Z_\s)$, uniformly in $\s$.	\vskip.15cm
	
	\item[{ (N5)}] Let \(I=[0,1]\), choose normalized arc-length parametrizations \(X_\s:I\to Z_\s^2\), and set \(p_\s=X_\s^{-1}\circ g_\s^1\circ g_\s^2\circ X_\s\). Then \(\|p_\s-\operatorname{id}_I\|_{C^1(I)}\leq C\s\), with \(C\) independent of \(\s\).\vskip.15cm
	
	\item[{ (N6)}] For every \(x\in(0,1)\), there is \(c_x>0\), independent of \(\s\), such that \(|p_\s(x)-x|\geq c_x\s\).
\end{enumerate}
In this setup,
the maps \(g_\s^i\) act on the arcs \(Z_\s^i\), rather than on \(S^1\). This is the essential distinction from the graph setting, where a coordinate projection identifies each component globally with \(S^1\).

\begin{lemma}[Alternating orbit on the constraint curve]
\label{lem:curve-orbit}
Assume \textup{(N1)--(N3)}. For every \(\g\in Z_\s^2\), there is a unique bi-infinite sequence \(z_k=(\theta_k,\theta_{k+1})\) such that \(z_0=\g\), \(z_{2k}\in Z_\s^2\), \(z_{2k+1}\in Z_\s^1\), and
\be
z_{2k+1}=g_\s^2(z_{2k}),\qquad
z_{2k+2}=g_\s^1(z_{2k+1}).
\ee
The associated angle sequence is a zero-energy chain.
\end{lemma}

\begin{proof}
Forward and backward iteration is well-defined because \(g_\s^1\) and \(g_\s^2\) are diffeomorphisms. The identity
\(\pi_1\circ g_\s^i=\pi_2|_{Z_\s^i}\) ensures that consecutive pairs share the same intermediate angle. Since every \(z_k\) lies in \(Z_\s\), every bond has zero potential energy.
\end{proof}

Assume, in addition, that the parametrizations \(X_\s\) in \textup{(N5)} converge uniformly to an embedding \(X_0:I\to M\), and that the normalized second iterates
\be
F_\s:=\frac{p_\s-\operatorname{id}_I}{2\s}
\ee
converge uniformly on \(I\) to \(F\). For initial points \(\g_\s\in Z_\s^2\) with \(X_\s^{-1}(\g_\s)\to s_0\), let \(s_\s(2k\s)=X_\s^{-1}(z_{2k})\) and interpolate linearly. By \textup{(N5)}, \(F_\s\) is uniformly Lipschitz, so the proof of Theorem~\ref{thm:one-periodic-limit} gives \(s_\s\to s\) locally uniformly, where \(s'=F(s)\) and \(s(0)=s_0\). By \textup{(N6)}, \(|F(x)|\geq c_x/2\) for every \(x\in(0,1)\); thus \(F\) has no interior zero and has a fixed sign there. The embedded even-pair path \(X_\s(s_\s(t))\) then converges to \(X_0(s(t))\) on the limiting constraint arc. Convergence of the odd-pair path additionally requires convergence of \(g_\s^2\circ X_\s\). In the non-monotone case, the limiting arc---not either angle coordinate---is the natural state space.

  
\section{Two-component geometry and continuum consistency for the rod linkage}
\label{sec:4}

\subsection{Asymptotic geometry of the two branches}

We now specialize our abstract two-component study to the physical linkage. First, we prove an \emph{implicit-function description} of the two zero-set branches in the asymptotic spinner regime, and an \emph{exact even--odd rewriting of the lattice equation.}

Fix \(\thetab\in(0,\pi/2)\), write \(R=\rt\), and set \(\s=R^{-1}\). Here \(\Lbf_R\) denotes the dimensionless spring length \(\Lbf\) in \eqref{equa--214} with \(\rt=R\), and \(\lt_R=\sqrt{1+4R^2\cos^2\thetab}\). Dividing the zero-energy constraint by \(2R^2\) gives
\bel{eq:normalized-zero-set}
F_\s(x,y):=1-\cos(y-x)-2\cos^2\thetab
+\s(\sin y-\sin x)=0.
\ee

\begin{proposition}[Limiting spinner branches]\label{prop:spinner-branches}
There exists \(\s_0=\s_0(\thetab)>0\) such that, for \(0\leq\s<\s_0\), the set \(F_\s^{-1}(0)\) is the disjoint union of two smooth graphs
\be
y=g_\s^-(x),\qquad y=g_\s^+(x).
\ee
For every integer \(k\geq0\), with the difference taken in compatible real lifts, one has 
\be
\big\|g_\s^\pm-(\operatorname{id}_{S^1}+\d_\pm)\big\|_{C^k(S^1)}
\leq C_k\s,
\qquad
\d_-=\pi-2\thetab,\quad \d_+=\pi+2\thetab.
\ee
In particular, \(\d_-+\d_+=2\pi\), so either alternating composition converges to the identity as \(\s\to0\).

For \(0<\s<\s_0\), let
\be
\Gbf_\s(x,y)=\frac{\kappa_\s}{2}
\big(\Lbf_R(x,y)-\lt_R\big)^2.
\ee
If \(0<c\leq\kappa_\s R^2\leq C\), then the nonzero eigenvalue of
\(D^2\Gbf_\s\) on the zero set is bounded above and below uniformly for small \(\s\).
\end{proposition}

\begin{proof}
\bse
At \(\s=0\), the equation \eqref{eq:normalized-zero-set} reduces to
\be
\cos(y-x)=-\cos(2\thetab),
\ee
whose two solutions are \(y-x=\d_\pm\) modulo \(2\pi\). Since
\be
\del_yF_0(x,x+\d_\pm)=\sin\d_\pm=\mp\sin(2\thetab)\neq0,
\ee
the uniform implicit-function theorem on \(S^1\) gives the two graphs and their \(C^k\) estimates. Uniform convergence of \(F_\s\) to \(F_0\) on the compact torus excludes any additional zeros away from these graphs. For the Hessian statement, observe that on the zero set
\be
D^2\Gbf_\s=\kappa_\s\,\nabla\Lbf_R\otimes\nabla\Lbf_R.
\ee
Moreover, \(R^{-1}\nabla\Lbf_R\) converges uniformly on each limiting branch to a nonzero vector. Differentiating \eqref{equa--214} and using \(\Lbf_R=\lt_R\) on the zero set gives the limits \((-\sin\thetab,\sin\thetab)\) on \(g_\s^-\) and \((\sin\thetab,-\sin\thetab)\) on \(g_\s^+\), uniformly in \(x\). The unique nonzero eigenvalue is therefore comparable to \(\kappa_\s R^2\).
\ese
\end{proof}

\subsection{Exact even--odd decomposition}

Let
\be
u_m:=\theta_{2m},\qquad v_m:=\theta_{2m+1}.
\ee
The following rewriting contains no approximation.

\begin{lemma}
\label{lem:even-odd-system}
The discrete Euler--Lagrange equation \eqref{eq:EL} is equivalent to the two-cell lattice system
\bel{eq:even-odd-system}
\begin{aligned}
\ddot u_m
&+\del_2\Gbf(v_{m-1},u_m)+\del_1\Gbf(u_m,v_m)=0,\\
\ddot v_m
&+\del_2\Gbf(u_m,v_m)+\del_1\Gbf(v_m,u_{m+1})=0.
\end{aligned}
\ee
\end{lemma}

\begin{proof}
Apply \eqref{eq:EL} first with \(k=2m\) and then with \(k=2m+1\), and use the definitions of \(u_m\) and \(v_m\).
\end{proof}

Let \(h=2a\) be the physical distance between sites of the same parity and let \(q>0\) be a time-scale factor. Both parity fields are indexed by the even-site pivot coordinate \(x=mh\). For smooth fields \(u=u(x,T)\) and \(v=v(x,T)\), we sample
\be
u_m(t)=u(mh,qt),\qquad v_m(t)=v(mh,qt)
\ee
and set also 
\be
H(u,v):=\del_2\Gbf(v,u)+\del_1\Gbf(u,v).
\ee

\needspace{9\baselineskip}
\begin{proposition}[Two-component Taylor residual]\label{prop:two-component-residual}
Assume \(u,v\in C^3\) with bounded derivatives on a compact space--time region and \(\Gbf\in C^4\) on the corresponding range. For \(h\)-dependent fields or potentials, assume that all pair arguments lie in a common compact subset of their smooth domain and that the stated derivative bounds are uniform as \(h\to0\). Then the residuals of \eqref{eq:even-odd-system} satisfy
\bel{eq:two-component-residual}
\begin{aligned}
\mathcal R_u={}&q^2\del_T^2u+H(u,v)
-h\,\del_{12}\Gbf(v,u)\del_xv\\
&+\frac{h^2}{2}\Big(
\del_{12}\Gbf(v,u)\del_x^2v
+\del_{112}\Gbf(v,u)(\del_xv)^2\Big)+O(h^3),\\
\mathcal R_v={}&q^2\del_T^2v+H(v,u)
+h\,\del_{12}\Gbf(v,u)\del_xu\\
&+\frac{h^2}{2}\Big(
\del_{12}\Gbf(v,u)\del_x^2u
+\del_{122}\Gbf(v,u)(\del_xu)^2\Big)+O(h^3),
\end{aligned}
\ee
uniformly on compact sets.
\end{proposition}

\begin{proof}
In the first equation of \eqref{eq:even-odd-system}, we expand
$
\del_2\Gbf\big(v(x-h,T),u(x,T)\big)
$
through second order in \(h\), while, in the second equation, we expand
$
\del_1\Gbf\big(v(x,T),u(x+h,T)\big).
$
The two force terms have the expansions
\be
\begin{aligned}
\del_2\Gbf(v(x-h),u)
={}&\del_2\Gbf(v,u)-h\del_{12}\Gbf(v,u)\del_xv\\
&+\frac{h^2}{2}\big(\del_{12}\Gbf(v,u)\del_x^2v
+\del_{112}\Gbf(v,u)(\del_xv)^2\big)+O(h^3),\\
\del_1\Gbf(v,u(x+h))
={}&\del_1\Gbf(v,u)+h\del_{12}\Gbf(v,u)\del_xu\\
&+\frac{h^2}{2}\big(\del_{12}\Gbf(v,u)\del_x^2u
+\del_{122}\Gbf(v,u)(\del_xu)^2\big)+O(h^3).
\end{aligned}
\ee
The two displayed Taylor formulas give \eqref{eq:two-component-residual}; the remaining terms are uniform because the stated derivatives are bounded.
\end{proof}

If \(q=h\), division by \(h^2\) produces the singular first-order terms \(h^{-1}\del_xu\), \(h^{-1}\del_xv\) and the algebraic terms \(h^{-2}H\). Their cancellation is not automatic. A formal expansion near a limiting spinner branch may use
\bel{eq:selected-branch-ansatz}
v=u+\d_\pm+h\lambda.
\ee
\needspace{9\baselineskip}
Two separate requirements then arise.
\begin{itemize}
\item The offset \(\d_\pm\) solves only the limiting constraint: \(F_\s(x,x+\d_\pm)=\s(\sin(x+\d_\pm)-\sin x)\). Thus, for fields varying over an interval, \eqref{eq:selected-branch-ansatz} requires a joint scaling \(\s=O(h)\); this condition alone does not cancel the singular forces.
\item The forces from both adjacent bonds must cancel compatibly. The nondegenerate normal Hessian in Proposition~\ref{prop:spinner-branches} alone does not establish that cancellation or control the normal correction \(\lambda\).
\end{itemize}
Substitution of \eqref{eq:selected-branch-ansatz} into the difference of the two equations may constrain \(\lambda\) when the normal Hessian from Proposition~\ref{prop:spinner-branches} is nondegenerate and the bond forces admit a compatible expansion. The averaged equation then may yield a scalar continuum model. Establishing that the discrete dynamics preserves \eqref{eq:selected-branch-ansatz} is precisely the content of Question~\ref{question:persistence}, stated earlier.

Proposition~\ref{prop:two-component-residual} rigorously identifies the local Taylor residual for prescribed smooth fields. The branch reduction \eqref{eq:selected-branch-ansatz}, the resulting scalar evolution equation, and convergence of discrete solutions require a separate stability estimate.

  
\section{Discrete standing configurations and branch selection}
\label{sec:5}

{ 
\subsection{Constraint curve and its projections}

A configuration \(\{\theta_k\}_{k\in\Z}\) has zero potential energy if and only if (see \eqref{equa--214}) 
\be
\Lbf(\theta_k,\theta_{k+1})=\lt
\qquad\hbox{for every } k
\ee
or, equivalently,
\bel{eq:discrete}
f(\theta_k,\theta_{k+1})=0,
\ee
with
\bel{eq:constraint-f}
\begin{aligned}
f(x,y)
&:=\Lbf(x,y)^2-\lt^2\\
&=2\rt(\sin y-\sin x)
+2\rt^2\big(1-\cos(y-x)\big)
-4\rt^2\cos^2\thetab.
\end{aligned}
\ee
We write \(Z=f^{-1}(0)\subset(S^1)^2\). Recall that Proposition~\ref{lem:period-two-states} lists the four points of \(Z\) that generate period-two configurations. We determine when the constraint is regular.

\begin{lemma}
\label{lem:regularity-zero-set}
The value \(0\) is regular for \(f\) if and only if
\be
\rt\sin^2\thetab\neq1,
\ee
or equivalently \(d\neq2/\sin\thetab\). At equality, the only critical point on \(Z\) is \((x,y)=(\pi/2,3\pi/2)\) modulo \(2\pi\).
\end{lemma}

\begin{proof}
\bse
The gradient is
\be
\nabla f(x,y)=
\begin{pmatrix}
-2\rt\cos x-2\rt^2\sin(y-x)\\
 2\rt\cos y+2\rt^2\sin(y-x)
\end{pmatrix}.
\ee
If \(\nabla f=0\), then \(\cos x=\cos y\). The alternative \(y=x\) is impossible on \(Z\). Hence \(y=-x\) modulo \(2\pi\), and the first component of the gradient gives
\be
\cos x\big(1-2\rt\sin x\big)=0.
\ee
If \(\sin x=(2\rt)^{-1}\), direct substitution gives
\(f(x,-x)=-1-4\rt^2\cos^2\thetab<0\). If \(\cos x=0\), only
\((x,y)=(\pi/2,3\pi/2)\) can lie on \(Z\), and there
\be
f(\pi/2,3\pi/2)=4\rt\big(\rt\sin^2\thetab-1\big). \qedhere
\ee 
\ese
\end{proof}

For the projections of \(Z\), let us define
\bel{eq:projection-threshold}
S_{\rt}:=\sqrt{1+4\rt^2\cos^2\thetab}
-2\rt\cos^2\thetab.
\ee

\begin{lemma}[Projection discriminant]\label{lem:projection-discriminant}
For fixed \(x\), the equation \(f(x,y)=0\) is equivalent to
\bel{eq:line-circle}
(1-\rt\sin x)\sin y-\rt\cos x\cos y
=\sin x+\rt\cos(2\thetab).
\ee
Its discriminant is
\bel{eq:projection-discriminant}
\begin{aligned}
D(x)
&=\cos^2x-4\rt\cos^2\thetab\sin x
+\rt^2\sin^2(2\thetab)\\
&=(S_{\rt}-\sin x)
\big(\sin x+S_{\rt}+4\rt\cos^2\thetab\big).
\end{aligned}
\ee
More precisely, \(D(x)>0\), \(D(x)=0\), or \(D(x)<0\) according as there are two, one, or no values of \(y\).

If \(\rt\sin^2\thetab>1\), then \(S_{\rt}>1\), \(D(x)>0\) for every \(x\), and \(Z\) is the disjoint union of two graphs over either coordinate circle. If \(\rt\sin^2\thetab<1\), then \(-1<S_{\rt}<1\), and there are unique
\be
-\frac{\pi}{2}<\theta_-<\frac{\pi}{2}
<\theta_+<\frac{3\pi}{2}
\ee
with \(\sin\theta_\pm=S_{\rt}\). There are no successors for
\(x\in(\theta_-,\theta_+)\), one at the endpoints, and two elsewhere.
\end{lemma}

\begin{proof}
\bse
Expanding \(\cos(y-x)\) in \eqref{eq:constraint-f} gives
\eqref{eq:line-circle}. The squared length of its left-hand coefficient vector minus the square of its right-hand side is
\be
D(x)=1-2\rt\sin x+\rt^2
-\big(\sin x+\rt\cos(2\thetab)\big)^2.
\ee
Factoring this quadratic in \(\sin x\) gives \eqref{eq:projection-discriminant}; its second root is
\(-S_{\rt}-4\rt\cos^2\thetab<-1\).
A line intersects the unit circle in two, one, or no points according to the sign of \(D\). The coefficient vector vanishes only at \(\rt=1\), \(x=\pi/2\); there the right-hand side is \(2\cos^2\thetab>0\), so \(D<0\) and there is no solution. Next, squaring positive quantities shows
\be
S_{\rt}>1\quad\Longleftrightarrow\quad
\rt\sin^2\thetab>1.
\ee
In this regime \(\rt>1\), so the coefficient vector in \eqref{eq:line-circle} never vanishes. The two intersection points are distinguished by the sign of their oriented determinant with that vector; each choice therefore defines a smooth graph over the entire \(x\)-circle.
In the opposite strict regime, the same calculation gives
\(-1<S_{\rt}<1\). The remaining assertions follow from the level sets of \(\sin x\). The equation is symmetric under
\((x,y)\mapsto(-y,-x)\), so the same conclusions hold for the other projection.
\ese
\end{proof}
 
\subsection{Multipliers of the period-two recurrence}
The multipliers below concern iteration of the stationary zero-energy constraint, not stability for the time-dependent Euler--Lagrange equation.

\begin{lemma}[Two-step multipliers]\label{lem:two-step-multipliers}
In the spinner regime, label the two graph components so that
\be
(\thetab,\pi-\thetab),(\pi+\thetab,-\thetab)\in Z_1,
\ee
and let \(g_i\) be the successor graph associated with \(Z_i\). For the alternating recurrence, set \(P=g_2\circ g_1\). At its two fixed points,
\be
\begin{aligned}
|P'(\thetab)|
&=\left|
\frac{\cos\thetab+\rt\sin(2\thetab)}
{\cos\thetab-\rt\sin(2\thetab)}
\right|^2>1,\\
|P'(\pi+\thetab)|
&=\left|
\frac{\cos\thetab-\rt\sin(2\thetab)}
{\cos\thetab+\rt\sin(2\thetab)}
\right|^2<1.
\end{aligned}
\ee
Thus \(\thetab\) is repelling and \(\pi+\thetab\) is attracting for the selected forward recurrence.
\end{lemma}

\begin{proof}
\bse
For large \(\rt\), the two components converge to the offset graphs
\(y=x+\pi\mp2\thetab\); hence the two displayed period-two points lie on the same component. This pairing persists throughout the connected spinner interval because Lemmas~\ref{lem:regularity-zero-set} and \ref{lem:projection-discriminant} prevent the two graph components from meeting.

Implicit differentiation of \(f(x,g(x))=0\) gives
\be
g'(x)=-\frac{\del_1f(x,g(x))}{\del_2f(x,g(x))}.
\ee
Applying this identity first to \(g_1\) and then to \(g_2\), and evaluating at the four points in \eqref{eq:period-two-states}, gives the derivatives
\be
\begin{aligned}
g_1'(\thetab)=g_2'(\pi-\thetab)
&=\frac{\cos\thetab+\rt\sin(2\thetab)}{\rt\sin(2\thetab)-\cos\thetab},\\
g_1'(\pi+\thetab)=g_2'(-\thetab)
&=\frac{\rt\sin(2\thetab)-\cos\thetab}{\rt\sin(2\thetab)+\cos\thetab}.
\end{aligned}
\ee
Their products give the reciprocal multipliers in the statement. In the spinner regime
\(\rt\sin^2\thetab>1\), so their strict inequalities follow.
\ese
\end{proof}

\subsection{Spinner standing configurations}

The following theorem is the main exact construction for the physical linkage. Its uniqueness is relative to the alternating component class; unrestricted branch switching is not excluded.
}
{ 
\begin{theorem}[Spinner heteroclinic family]\label{thm:spinners}
Assume \(\rt\sin^2\thetab>1\), equivalently \(d>2/\sin\thetab\). For every
\be
\g\in S^1\setminus\{\thetab,\thetab+\pi\},
\ee
there exists a zero-energy chain \(\{\theta_k\}_{k\in\Z}\) with \(\theta_0=\g\) and
\bel{eq:spinner-asymptotics}
\begin{aligned}
\lim_{k\to+\8}\theta_{2k}&=\thetab+\pi,
&\qquad \lim_{k\to-\8}\theta_{2k}&=\thetab,\\
\lim_{k\to+\8}\theta_{2k+1}&=-\thetab,
&\qquad \lim_{k\to-\8}\theta_{2k+1}&=\pi-\thetab.
\end{aligned}
\ee
More precisely:
\begin{enumerate}
\item \(Z=f^{-1}(0)\) has two connected components \(Z_1,Z_2\), with
\((\thetab,\pi-\thetab),(\pi+\thetab,-\thetab)\in Z_1\) and
\((\pi-\thetab,\thetab),(-\thetab,\pi+\thetab)\in Z_2\).
\item The chain satisfies
\be
(\theta_{2k},\theta_{2k+1})\in Z_1,\qquad
(\theta_{2k+1},\theta_{2k+2})\in Z_2
\ee
for every \(k\).
\item It is the unique chain with \(\theta_0=\g\) satisfying the component condition in item \(2\).
\item Each parity subsequence is strictly monotone on the corresponding lifted arc of \(S^1\).
\item For fixed \(k\), the map \(\g\mapsto\theta_k\) is monotone on each component of
\(S^1\setminus\{\thetab,\thetab+\pi\}\).
\end{enumerate}
\end{theorem}
}

\begin{proof}
\bse
By Lemmas~\ref{lem:regularity-zero-set} and \ref{lem:projection-discriminant}, \(Z\) is regular and has exactly two components, each of which is a graph over either coordinate circle. In the limit \(\rt\to\8\), these components converge to
\be
y=x+\pi-2\thetab,\qquad y=x+\pi+2\thetab.
\ee
After division of \eqref{eq:constraint-f} by \(2\rt^2\), the roots at \(\rt^{-1}=0\) are simple because \(\sin(2\thetab)>0\). The implicit-function theorem therefore gives uniform \(C^1\) convergence of the two graph maps to these offsets.
Since the spinner interval \((\sin^{-2}\thetab,\8)\) is connected and the components never meet there, the first and fourth points in \eqref{eq:period-two-states} lie on one component, denoted by \(Z_1\), and the other two lie on \(Z_2\). This proves item \(1\).

For \((x,y)\in Z_1\), let \(m(x,y)\) be the unique point \((y,z)\in Z_2\). Define \(m\) analogously from \(Z_2\) to \(Z_1\). The graph property gives existence and uniqueness; the same property for the other projection gives a smooth inverse. Hence \(m:Z\to Z\) is a diffeomorphism interchanging \(Z_1\) and \(Z_2\).

{ 
By Proposition~\ref{lem:period-two-states}, the fixed points of \(m^2\) on \(Z_1\) are
\be
p_-=(\thetab,\pi-\thetab),\qquad
p_+=(\pi+\thetab,-\thetab),
\ee
and the two fixed points on \(Z_2\) are their images under \(m\). The complement
\(Z_1\setminus\{p_-,p_+\}\) has two open arcs. For large \(\rt\), the map \(m^2\) is \(C^1\)-close to the identity because the limiting offsets sum to \(2\pi\); hence it is orientation preserving. Its degree cannot change while \(m\) remains a diffeomorphism. Moreover, a fixed point of \(m^2\) is exactly a pair generating a period-two zero-energy chain, and Proposition~\ref{lem:period-two-states} lists all such pairs. Hence no interior fixed point can appear anywhere in the connected spinner regime. Since \(m^2\) leaves both endpoints invariant and is orientation preserving, it preserves each open arc and is a strictly increasing interval map there.

Given \(\g\), let \(z_0=(\g,\theta_1)\) be the unique point of \(Z_1\) with first coordinate \(\g\), and set
\be
\theta_k:=\pi_1\big(m^k(z_0)\big),\qquad k\in\Z.
\ee
This defines a zero-energy chain and proves items \(2\) and \(3\). Since an increasing interval map without an interior fixed point lies strictly on one side of the identity, the even iterates are strictly monotone and converge to the two endpoints of the lifted arc. Lemma~\ref{lem:two-step-multipliers} identifies the attracting endpoint in forward iteration as \(p_+\). Hence
\be
\lim_{k\to+\8}\theta_{2k}=\pi+\thetab,\qquad
\lim_{k\to-\8}\theta_{2k}=\thetab.
\ee
The same argument for \(m^2|_{Z_2}\), followed by one application of \(m\), gives the strict monotonicity and the two odd-site limits in \eqref{eq:spinner-asymptotics}. Because each iterate of the relevant graph diffeomorphisms preserves or reverses order globally on a lifted arc, the map \(\g\mapsto\theta_k\) is monotone there for every fixed \(k\), which proves item \(5\).
}
\ese
\end{proof}

\subsection{An explicit successor formula}

{ 
The proof of Theorem~\ref{thm:spinners} selects successors by their zero-set component. The following formula gives the same two candidates in a local chart; Section~\ref{sec:numerical-spinner} evaluates the nonsingular line--circle equation directly.

\begin{proposition}[Explicit successor roots]\label{prop:explicit-successor}
Let \(x\in S^1\) satisfy \(1-\rt\sin x\neq0\), and set
\be
\label{eq:explicit-successor-coefficients}
\alpha(x)=\frac{\rt\cos x}{1-\rt\sin x},
\qquad
\delta(x)=\frac{\sin x+\rt\cos(2\thetab)}{1-\rt\sin x}.
\ee
Real successors \(y\in S^1\) satisfying \(f(x,y)=0\) exist if and only if
\(1+\alpha(x)^2-\delta(x)^2\geq0\). When this condition holds, they are obtained from
\be
\label{eq:explicit-successor-roots}
\begin{aligned}
\cos y&=\frac{-\alpha(x)\delta(x)\pm
\sqrt{1+\alpha(x)^2-\delta(x)^2}}{1+\alpha(x)^2},\\
\sin y&=\alpha(x)\cos y+\delta(x).
\end{aligned}
\ee
In the spinner regime, the component condition in Theorem~\ref{thm:spinners} selects one of these roots at each step; the selected components alternate, while the \(\pm\) label is local to this chart. At a point where \(1-\rt\sin x=0\), the selected successor is obtained from the nonsingular line--circle equation \eqref{eq:line-circle} and continuity along that component.
\end{proposition}

\begin{proof}
Equation~\eqref{eq:line-circle} is equivalent in this chart to
\(\sin y=\alpha(x)\cos y+\delta(x)\). Substitution into
\(\sin^2y+\cos^2y=1\) gives a quadratic equation for \(\cos y\), whose two roots are exactly \eqref{eq:explicit-successor-roots}. Theorem~\ref{thm:spinners} identifies the two roots with the two graph components. On each connected interval of this chart, one sample point on each component fixes its local sign; continuity then preserves that sign until the alternate chart is required.
\end{proof}
}

\subsection{Flipper constraint geometry}

{  In the flipper regime, the coordinate projections of the constraint curve have folds, and the zero-energy recurrence is two-valued. We begin by classifying the global topology and the fibers of \(Z=\{f(x,y)=0\}\).}
The theorem below determines the flipper constraint geometry and the number of local successors; global branch selection is stated separately as Conjecture~\ref{conj:flipper-selection}.

{ 
\begin{theorem}[Geometry of the flipper zero set]
\label{prop:flipper-geometry}
Assume
\be
\rt\sin^2\thetab<1,
\qquad\hbox{equivalently}\qquad
d<\frac{2}{\sin\thetab}.
\ee
Then \(Z\) is a smooth embedded circle that is null-homotopic in \(M\). Let \(S_{\rt}\) be defined by \eqref{eq:projection-threshold}, and let
\be
\theta_-=\arcsin S_{\rt},\qquad
\theta_+=\pi-\arcsin S_{\rt}.
\ee
For \(x\in(\theta_-,\theta_+)\), there is no \(y\) with \((x,y)\in Z\); for
\(x=\theta_\pm\), there is exactly one; and for
\(x\in S^1\setminus[\theta_-,\theta_+]\), there are exactly two.
Moreover, \(\del_yf\) vanishes on \(Z\) precisely at the two fold points with
\(x\)-coordinates \(\theta_\pm\). The analogous statement holds after exchanging \((x, y)\) with \((-y, -x)\).
\end{theorem}

\begin{proof}
Lemma~\ref{lem:regularity-zero-set} gives smoothness, while
Lemma~\ref{lem:projection-discriminant} gives the stated number of points in each fiber. Over the interior of the admissible \(x\)-arc, the two roots depend smoothly on \(x\); they meet at both endpoints and exhaust the zero set. Consequently, \(Z\) is one embedded circle. Its projection onto either coordinate is a proper arc of \(S^1\), so \(Z\) lies in the product of two proper arcs, a disk in \(M\), and is therefore null-homotopic in \(M\).

The line in \eqref{eq:line-circle} is tangent to the unit circle exactly when its discriminant vanishes. This tangency is equivalent to
\(\del_yf=0\), and \eqref{eq:projection-discriminant} shows that it occurs only for \(x=\theta_\pm\).
At these points \(\cos\theta_\pm\neq0\), so the coefficient vector in \eqref{eq:line-circle} is nonzero. Its tangency to the unit circle gives \(\del_{yy}f\neq0\); regularity of \(Z\) then gives \(\del_xf\neq0\). Thus both critical points are nondegenerate folds.
Symmetry of the entire problem under
\((x,y)\mapsto(-y,-x)\) gives the statement for the other projection.
\end{proof}
}

{ 
 
\subsection{Global branch selection in the flipper regime}

For \(z=(x,y)\in Z\), introduce the successor correspondence
\be
\mathcal T(z):=\{(y,w)\in Z\}.
\ee
Let
\be
\begin{aligned}
p_-&=(\thetab,\pi-\thetab),
& q_-&=(\pi-\thetab,\thetab),\\
p_+&=(-\thetab,\pi+\thetab),
& q_+&=(\pi+\thetab,-\thetab).
\end{aligned}
\ee

\begin{conjecture}[Flipper branch selection]\label{conj:flipper-selection}
In the non-wobbling flipper regime $d < 1$, there are closed arcs \(I_1,I_2\subset Z\), with endpoints
\be
\del I_1=\{p_-,p_+\},\qquad
\del I_2=\{q_-,q_+\},
\ee
and single-valued selections
\be
T_1:I_1\to I_2,\qquad T_2:I_2\to I_1,
\qquad T_i(z)\in\mathcal T(z),
\ee
that are diffeomorphisms on the interiors. The second iterate
\(P=T_2\circ T_1\) is orientation preserving, has only \(p_-\) and \(p_+\) as fixed points, and moves every interior point from \(p_-\) toward \(p_+\).

Consequently, every initial pair in the interior of \(I_1\) generates a unique selected heteroclinic zero-energy chain. Moreover, its even-site angle can be chosen monotone.
\end{conjecture}

\begin{question}[A checkable geometric criterion]\label{question:flipper-sign}
Can the existence and ordering of \(I_1,I_2\) be proved directly from the line--circle representation \eqref{eq:line-circle}, for example by establishing a sign for the derivative of the two-step map between consecutive fold points? If the sign changes for some admissible parameters, what is the smallest counterexample to Conjecture~\ref{conj:flipper-selection}?
\end{question}

\begin{question}[Wobbling flippers]
Is there a simple geometric branch-selection mechanism in the wobbling flipper regime \(1<d<2/\sin\thetab\)? In this range, the monotonicity of even-site angles will not be monotone (this was already observed in \cite{CUV}), but moreover the diffeomorphism structure in Conjecture \ref{conj:flipper-selection} is not always possible.
\end{question}

The local geometry required to formulate Conjecture~\ref{conj:flipper-selection} follows from Theorem~\ref{prop:flipper-geometry}; the existence and global ordering of the selected arcs remain open. 
 
\subsection{Geometric and regime overview}

{  The following diagrams summarize the geometry and the parameter regimes used in the analysis. They are schematic and introduce no additional assumption.}

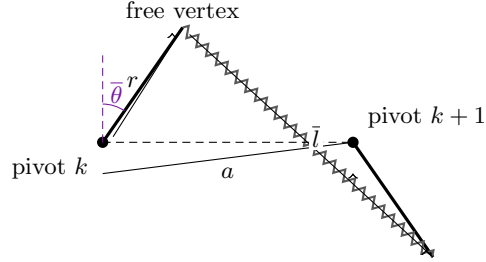
\begin{figure}[htbp]
\centering
\begin{tikzpicture}[scale=0.92, every node/.style={font=\small}]
\coordinate (A0) at (0,0);
\coordinate (A1) at (3.6,0);
\coordinate (B0) at (1.15,1.65);
\coordinate (B1) at (4.75,-1.65);
\fill (A0) circle (2.3pt);
\fill (A1) circle (2.3pt);
\draw[very thick] (A0)--(B0);
\draw[very thick] (A1)--(B1);
\draw[gray!70!black, thick, decorate, decoration={zigzag,segment length=5pt,amplitude=2pt}] (B0)--(B1);
\draw[densely dashed] (A0)--(A1);
\draw[<->] (A0)+(0,-0.45)--node[below] {$a$}(A1)+(0,-0.45);
\draw[<->] (A0)+(0.15,0.08)--node[left] {$r$}(B0)+(-0.15,-0.08);
\draw[<->] (B0)+(0.1,-0.1)--node[pos=0.52,above,fill=white,inner sep=1pt] {$\bl$}(B1)+(-0.1,0.1);
\draw[densely dashed,revisionviolet] (A0)--+(0,1.25);
\draw[revisionviolet] (0,0.55) arc[start angle=90,end angle=55,radius=0.55];
\node[revisionviolet] at (0.20,0.72) {$\thetab$};
\node[below left=2pt] at (A0) {pivot $k$};
\node[above] at (B0) {free vertex};
\node[above right=2pt] at (A1) {pivot $k+1$};
\end{tikzpicture}
\caption{{  Schematic geometry of two adjacent rods. The reference angle \(\thetab\) is measured from the vertical axis; the rods have length \(r\), the pivots are separated by \(a\), and the spring has relaxed length \(\bl\).}}
\label{fig:rod-geometry}
\end{figure}

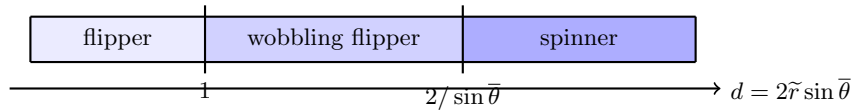
\begin{figure}[htbp]
\centering
\begin{tikzpicture}[x=1.1cm,y=1cm, every node/.style={font=\small}]
\fill[blue!8] (0,0.35) rectangle (2.1,0.95);
\fill[blue!18] (2.1,0.35) rectangle (5.2,0.95);
\fill[blue!32] (5.2,0.35) rectangle (8,0.95);
\draw[->,thick] (-0.25,0)--(8.3,0) node[right] {$d=2\rt\sin\thetab$};
\draw[thick] (0,0.35)--(8,0.35);
\draw[thick] (0,0.95)--(8,0.95);
\draw[thick] (0,0.35)--(0,0.95);
\draw[thick] (2.1,0.2)--(2.1,1.1);
\draw[thick] (5.2,0.2)--(5.2,1.1);
\draw[thick] (8,0.35)--(8,0.95);
\node[below] at (2.1,0.2) {$1$};
\node[below] at (5.2,0.2) {$2/\sin\thetab$};
\node at (1.05,0.65) {flipper};
\node at (3.65,0.65) {wobbling flipper};
\node at (6.6,0.65) {spinner};
\end{tikzpicture}
\caption{{  Parameter regimes used in Section~\ref{sec:2}. The boundary between the flipper and wobbling-flipper regimes is $d=1$, while the spinner threshold is $d=2/\sin\thetab$.}}
\label{fig:regimes}
\end{figure}

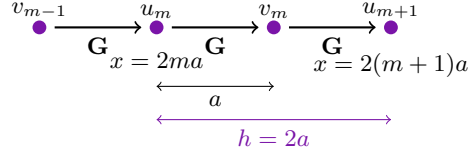
\begin{figure}[htbp]
\centering
\begin{tikzpicture}[x=1.55cm,y=1cm, every node/.style={font=\small}]
\foreach \x/\lab in {0/{v_{m-1}},1/{u_m},2/{v_m},3/{u_{m+1}}} {
  \fill[revisionviolet] (\x,0) circle (2.6pt);
  \node[above] at (\x,0) {$\lab$};
}
\draw[thick,->] (0.13,0)--(0.87,0) node[midway,below] {$\Gbf$};
\draw[thick,->] (1.13,0)--(1.87,0) node[midway,below] {$\Gbf$};
\draw[thick,->] (2.13,0)--(2.87,0) node[midway,below] {$\Gbf$};
\node[below=6pt] at (1,0) {$x=2ma$};
\node[below=6pt] at (3,0) {$x=2(m+1)a$};
\draw[<->] (1,-0.78)--node[below] {$a$}(2,-0.78);
\draw[<->,revisionviolet] (1,-1.22)--node[below] {$h=2a$}(3,-1.22);
\end{tikzpicture}
\caption{{  Exact two-cell decomposition from Lemma~\ref{lem:even-odd-system}. Adjacent sites are separated by \(a\), while equal-parity cells are separated by \(h=2a\).}}
\label{fig:two-component}
\end{figure}

} 


\section{Numerical experiments}
\label{sec:numerics}

\subsection{Discrete equations, common grid, and verification criteria}
\label{sec:numerical-method}

\subsubsection{Purpose}
The computations in this section serve three distinct purposes. 
\bei 
\item First, we construct a selected zero-energy spinner chain for the physical linkage and compare its asymptotic rate with Lemma~\ref{lem:two-step-multipliers}. 

\item Second, we evaluate the discrete residual on a prescribed smooth function in order to test the two consistency orders in Propositions~\ref{prop:first-order-consistency} and \ref{prop:second-order-consistency}. 
\item Third, we evolve the discrete, leading, and corrected models from identical smooth periodic data and measure their discrepancies. 
\eei

{  
\subsubsection{Discrete and continuum evolutions}
We use the simplified pair potential in \eqref{eq:simple-potential}, set \(\sigma=\beta a\), and work in the continuum time \(T=at\) used in Proposition~\ref{prop:first-order-consistency}. On the periodic grid \(x_j=ja\), \(0\leq j<N\), with \(a=2\pi/N\), we define
\be
\label{eq:num-bond-phase}
z_j=U_{j+1}-U_j-\beta a\cos U_j,
\ee
where all spatial indices are taken modulo \(N\). Direct differentiation of the pair potential gives the discrete equation
\be
\label{eq:num-discrete-evolution}
\del_T^2U_j=\frac{1}{2a^2}\Big(-\sin z_{j-1}+(1-\beta a\sin U_j)\sin z_j\Big).
\ee
The two continuum equations compared with \eqref{eq:num-discrete-evolution} are
\be
\label{eq:num-leading-evolution}
\del_T^2U=\frac12\del_x^2U+\frac{\beta^2}{4}\sin(2U)
\ee
and
\be
\label{eq:num-corrected-evolution}
\del_T^2U=\frac12\del_x^2U+\frac{\beta^2}{4}\sin(2U)-\frac{\beta a}{4}\Big(2\sin U\,\del_x^2U+\cos U\,(\del_xU)^2\Big).
\ee
The sign and derivative order used in the implementation were checked directly by differentiating the discrete potential. The independent finite-difference force test described below provides a second check of this calculation.

\subsubsection{Spatial and temporal discretization}
For both continuum equations, we use the centered operators
\be
\label{eq:numerical-difference-operators}
D_0U_j=\frac{U_{j+1}-U_{j-1}}{2a},\qquad D_2U_j=\frac{U_{j+1}-2U_j+U_{j-1}}{a^2}.
\ee
The resulting second-order systems are written in first-order form and advanced by the classical fourth-order Runge--Kutta method. For a final time \(T_*\), we set
\be
\label{eq:num-time-step}
M_N=\left\lceil\frac{T_*}{0.12a}\right\rceil,\qquad \Delta T=\frac{T_*}{2M_N}.
\ee
Every run is repeated with the coarser step \(2\Delta T\). The temporal discrepancy is measured between these two computations on the same spatial grid and is required to be substantially smaller than the discrepancy between models.
More precisely, for \(Y=(U,V)\) and \(F(Y)=(V,\mathcal A(U))\), where \(\mathcal A\) is the relevant right-hand side of \eqref{eq:num-discrete-evolution}, \eqref{eq:num-leading-evolution}, or \eqref{eq:num-corrected-evolution} with the spatial derivatives replaced by \eqref{eq:numerical-difference-operators}, we use
\be
\label{eq:num-rk4}
\begin{aligned}
K_1&=F(Y_n),&K_2&=F(Y_n+\tfrac12\Delta T K_1),\\
K_3&=F(Y_n+\tfrac12\Delta T K_2),&K_4&=F(Y_n+\Delta T K_3),\\
Y_{n+1}&=Y_n+\tfrac16\Delta T(K_1+2K_2+2K_3+K_4).
\end{aligned}
\ee
The coarse run uses \(M_N\) steps and the reported fine run uses \(2M_N\) steps; for \(T_*=1.5\) these fine step counts are \(256,510,1020,2038,4076\) as \(N\) runs through \(64,128,256,512,1024\). Spatial differences act on a consistent real lift of the angles; only the comparison norm reduces differences modulo \(2\pi\).

\subsubsection{Error and energy diagnostics}
Since all three computations use the same periodic nodes, we measure
\be
\label{eq:num-error-norms}
E_\infty(T)=\max_j\operatorname{dist}_{S^1}(U_j(T),V_j(T)),\qquad E_2(T)=\left(\frac1N\sum_{j=0}^{N-1}\operatorname{dist}_{S^1}(U_j(T),V_j(T))^2\right)^{1/2}.
\ee
Here \(\operatorname{dist}_{S^1}(p,q)=|((p-q+\pi)\bmod 2\pi)-\pi|\). The adjacent-mesh rates in the tables are \(\log(E_N/E_{2N})/\log2\); each fitted rate is the least-squares slope of \(\log E_N\) against \(\log a\) on the four finest meshes.
For the discrete chain, the conserved semidiscrete energy is
\be
\label{eq:num-discrete-energy}
\mathcal E_a(U,V)=a\sum_{j=0}^{N-1}\left(\frac12V_j^2+\frac1{a^2}\sin^2\!\left(\frac{z_j}{2}\right)\right).
\ee
We monitor the maximum relative drift of \(\mathcal E_a\) at more than two hundred times during each run. For a run with \(S=2M_N\) steps, the energy is sampled every \(\max(1,\lfloor S/200\rfloor)\) steps and at the endpoint; the reported drift is the maximum of \(|\mathcal E_a(T_n)-\mathcal E_a(0)|/\max(|\mathcal E_a(0)|,10^{-15})\) over these samples. The analogous leading-model energy is checked in the data file; no conserved energy is claimed for the corrected discretization. As an implementation check independent of time evolution, centered finite differences of the pair potential reproduce its two analytic derivatives with maximum absolute discrepancy \(4.4\times10^{-11}\).

\subsection{Selected spinner chain for the physical linkage}
\label{sec:numerical-spinner}

We first test the exact recurrence underlying Theorem~\ref{thm:spinners} and Proposition~\ref{prop:explicit-successor}. We choose
\be
\label{eq:num-spinner-parameters}
\thetab=\frac{\pi}{6},\qquad \rt=4.5,\qquad d=2\rt\sin\thetab=4.5>\frac{2}{\sin\thetab}=4,
\ee
and prescribe \(\theta_0=\pi/2\). For fixed \(x\), equation \eqref{eq:line-circle} has two explicit successor values. Writing
\be
\label{eq:num-successor-coefficients}
A(x)=1-\rt\sin x,\qquad B(x)=-\rt\cos x,\qquad C(x)=\sin x+\rt\cos(2\thetab),
\ee
we solve \(A(x)\sin y+B(x)\cos y=C(x)\) and select successively the root closest, in a consistent lift, to \(x+\pi-2\thetab\) and \(x+\pi+2\thetab\). The backward orbit is obtained from the symmetry \((x,y)\mapsto(-y,-x)\). This procedure evaluates the exact zero-energy constraint and introduces no minimization or time discretization.
Explicitly, set \(\rho=(A^2+B^2)^{1/2}\), \(\delta=\operatorname{atan2}(B,A)\), and \(\alpha=\arcsin(C/\rho)\). The candidates modulo \(2\pi\) are \(y=\alpha-\delta\) and \(y=\pi-\alpha-\delta\). At each even, respectively odd, index we lift both candidates near \(x+\pi-2\thetab\), respectively \(x+\pi+2\thetab\), and choose the closer one. We use 32 cells on either side of \(m=0\); reverse iteration applies the stated symmetry and the same nearest-lift rule.

Figure~\ref{fig:numerical-spinner} shows the two reduced parity components on \(-32\leq m\leq32\). The maximum value of \(|f(\theta_k,\theta_{k+1})|\) is \(1.11\times10^{-12}\). The oriented determinant \(\Delta(x,y):=A(x)\cos y-B(x)\sin y\) distinguishes the two line--circle intersections. On the computed even edges, \(\Delta\geq1.299\); on the odd edges, \(\Delta\leq-3.031\). These nonzero margins verify the alternating component assignment for this orbit. Lemma~\ref{lem:two-step-multipliers} predicts the attracting multiplier
\be
\label{eq:num-spinner-multiplier}
q_+=\left|\frac{\cos\thetab-\rt\sin(2\thetab)}{\cos\thetab+\rt\sin(2\thetab)}\right|^2=0.40495868.
\ee
A least-squares fit of the logarithm of the endpoint error over indices \(m\geq3\) in the range where it lies between \(10^{-12}\) and \(0.25\) gives \(q_{+,\mathrm{num}}=0.40492540\). The relative discrepancy is \(8.2\times10^{-5}\).

\begin{figure}[htbp]
\centering
\includegraphics[width=0.98\textwidth]{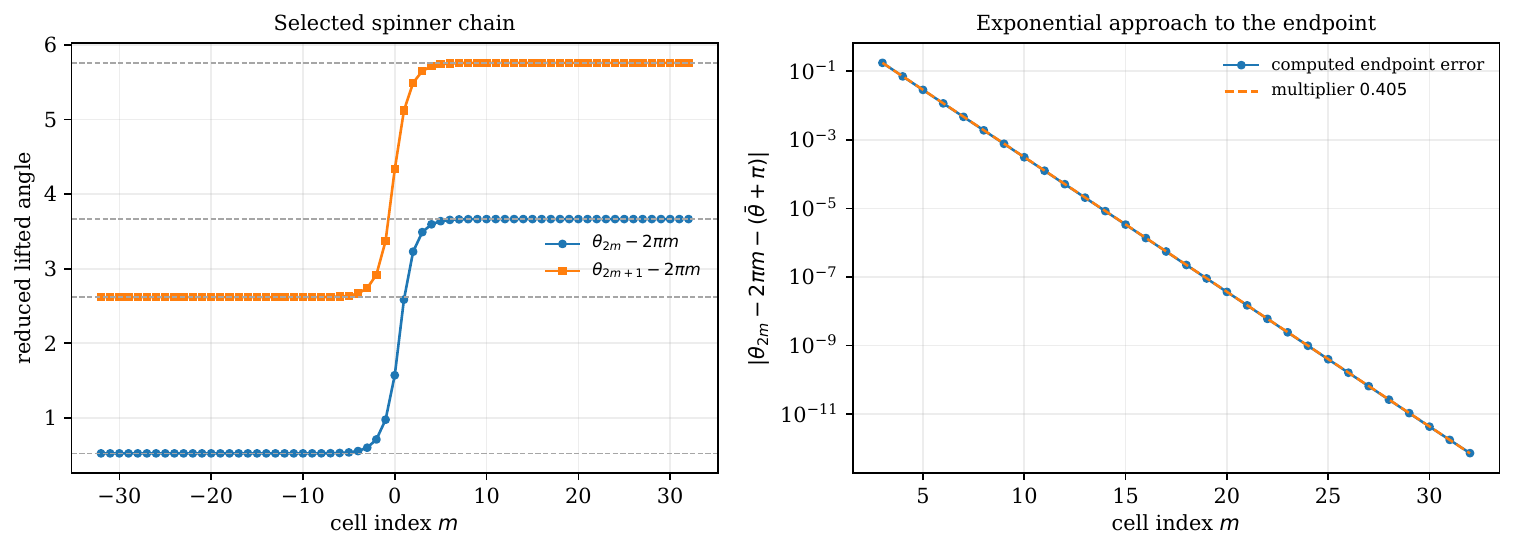}
\caption{Selected spinner chain for \(\thetab=\pi/6\), \(\rt=4.5\), and \(\theta_0=\pi/2\). Left: reduced lifted even and odd components and their four limiting states. Right: approach of the even component to its forward endpoint, together with the multiplier from \eqref{eq:num-spinner-multiplier}.}
\label{fig:numerical-spinner}
\end{figure}

This computation provides a reproducible orbit in the alternating component class of Theorem~\ref{thm:spinners} and an independent numerical check of Lemma~\ref{lem:two-step-multipliers}. It does not test orbital stability under the time-dependent rod-linkage equation.

\subsection{Manufactured residual and local consistency orders}
\label{sec:numerical-residual}

We next test the local expansions without introducing a numerical solution of either continuum equation. On \([0,2\pi]\), we prescribe the smooth periodic function
\be
\label{eq:num-manufactured-field}
u(x,T)=\frac\pi2+0.23\cos x\cos T+0.11\sin(2x)\sin(0.7T),
\ee
set \(\beta=1.4\), and evaluate all quantities at \(T=0.73\). The normalized discrete residual is formed by inserting the exact values of \(u\) and \(\del_T^2u\) into \eqref{eq:num-discrete-evolution}. We compare it first with the leading differential operator in Proposition~\ref{prop:first-order-consistency}, and then with the corrected operator in Proposition~\ref{prop:second-order-consistency}. Consequently, the experiment measures the consistency expansions themselves and is unaffected by temporal integration error.
The spatial derivatives \(u_x\) and \(u_{xx}\) in these comparison operators are evaluated analytically from \eqref{eq:num-manufactured-field}, not by \(D_0\) or \(D_2\). Thus the table isolates the Taylor residual of the exact lattice force from the truncation error of the continuum spatial discretization.

Table~\ref{tab:manufactured-residual} and Figure~\ref{fig:manufactured-residual} show the maximum discrepancies. The rates between the two finest meshes are \(1.002\) and \(1.988\); least-squares fits over the four finest meshes give \(1.004\) and \(1.971\). These values reproduce the orders \(O(a)\) and \(O(a^2)\) in the two propositions.

\begin{table}[htbp]
\centering
\caption{Maximum discrepancy between the normalized discrete residual and the leading or corrected differential operator for \eqref{eq:num-manufactured-field}.}
\label{tab:manufactured-residual}
\begin{tabular}{rccccc}
\hline
\(N\) & \(a\) & leading error & rate & corrected error & rate\\
\hline
64   & \(9.817\times10^{-2}\) & \(2.371\times10^{-2}\) & --    & \(4.540\times10^{-4}\) & --\\
128  & \(4.909\times10^{-2}\) & \(1.175\times10^{-2}\) & 1.013 & \(1.227\times10^{-4}\) & 1.888\\
256  & \(2.454\times10^{-2}\) & \(5.845\times10^{-3}\) & 1.007 & \(3.177\times10^{-5}\) & 1.949\\
512  & \(1.227\times10^{-2}\) & \(2.915\times10^{-3}\) & 1.004 & \(8.080\times10^{-6}\) & 1.975\\
1024 & \(6.136\times10^{-3}\) & \(1.456\times10^{-3}\) & 1.002 & \(2.037\times10^{-6}\) & 1.988\\
\hline
\end{tabular}
\end{table}

\begin{figure}[htbp]
\centering
\includegraphics[width=0.67\textwidth]{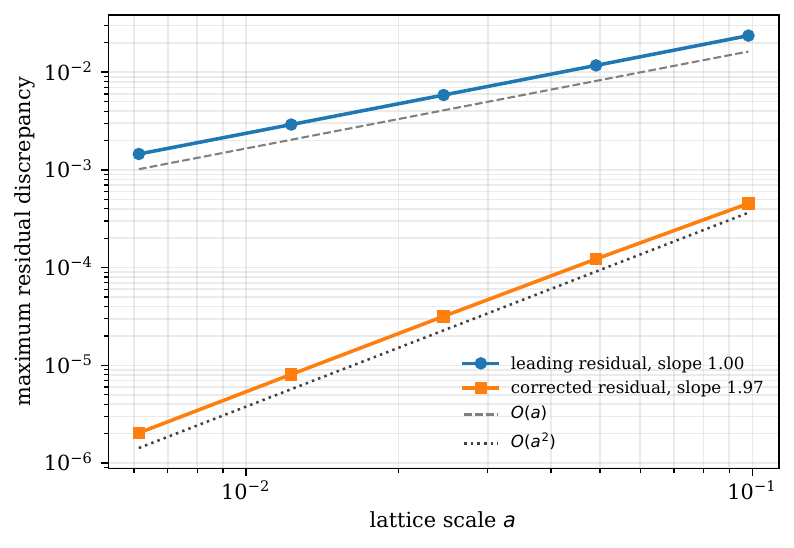}
\caption{Manufactured-residual test for the smooth field \eqref{eq:num-manufactured-field}. The observed slopes reproduce the one-power improvement predicted by Proposition~\ref{prop:second-order-consistency}.}
\label{fig:manufactured-residual}
\end{figure}

\subsection{Dynamic comparison from common periodic data}
\label{sec:numerical-dynamics}

We now solve the three semidiscrete systems through \(T_*=1.5\), with \(\beta=1.4\), from the common initial data
\be
\label{eq:num-initial-data}
\begin{aligned}
U(x,0)&=\frac\pi2+0.31\cos x+0.12\sin(2x),\\
\del_TU(x,0)&=0.08\sin x-0.04\cos(2x).
\end{aligned}
\ee
These data are smooth, periodic, and remain in one angular chart during the computed interval. The same sampled values are passed to \eqref{eq:num-discrete-evolution}, \eqref{eq:num-leading-evolution}, and \eqref{eq:num-corrected-evolution}; hence interpolation and boundary errors do not enter the comparison.
There is no random initialization. Each temporal check is the angular \(E_\infty\) discrepancy between the fine and coarse runs of the same model; the last column of Table~\ref{tab:dynamic-convergence} reports the largest of the three checks.

Figure~\ref{fig:dynamic-profiles} displays the three profiles for \(N=64\). The corrected profile is visually indistinguishable from the discrete one at this scale, while the right panel resolves their pointwise discrepancies. Table~\ref{tab:dynamic-convergence} gives \(E_\infty(T_*)\) and the temporal check. The leading discrepancy has an observed adjacent-mesh rate of \(1.005\) between the two finest meshes and a fitted rate of \(1.012\). The corrected discrepancy has an observed adjacent-mesh rate of \(1.997\) between the two finest meshes and a fitted rate of \(1.994\). At \(N=1024\), the corrected discrepancy is \(6.86\times10^{-7}\), whereas the largest temporal discrepancy among the three models is \(1.73\times10^{-14}\). The maximum relative drift of the discrete energy \eqref{eq:num-discrete-energy} over all five fine runs is \(1.59\times10^{-12}\).

\begin{figure}[htbp]
\centering
\includegraphics[width=0.98\textwidth]{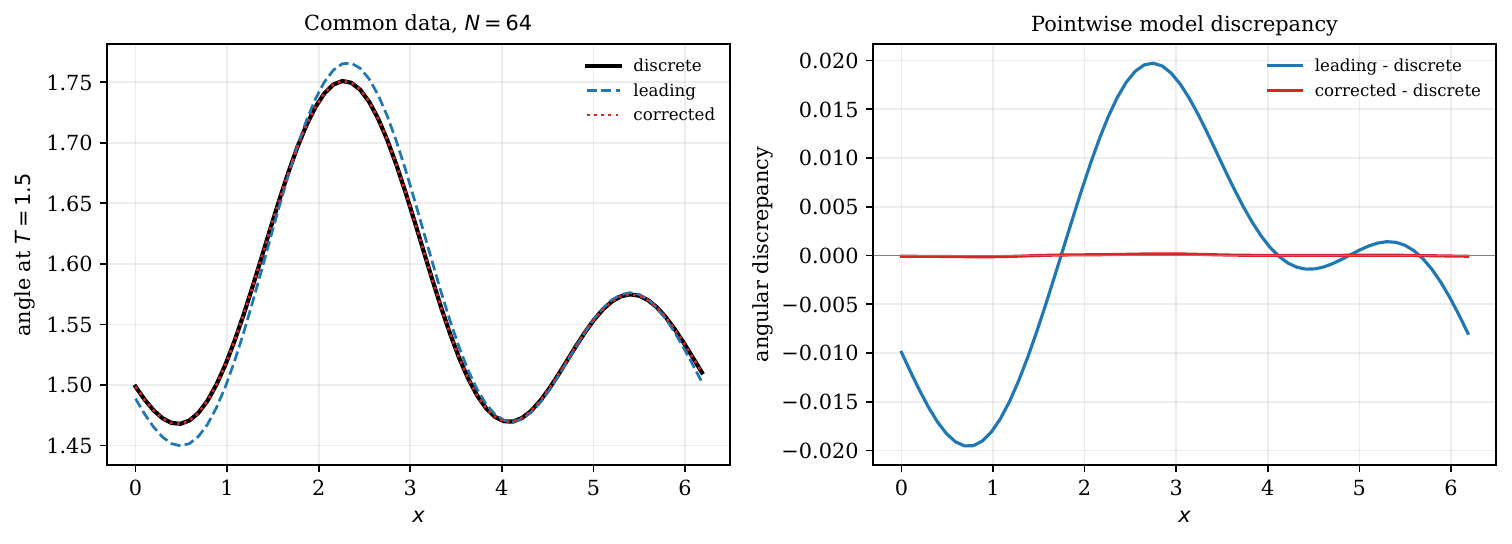}
\caption{Dynamic comparison at \(T=1.5\) for the common periodic data \eqref{eq:num-initial-data} and \(N=64\). Left: discrete, leading, and corrected profiles. Right: angular discrepancy of each continuum model relative to the discrete chain.}
\label{fig:dynamic-profiles}
\end{figure}

\begin{table}[htbp]
\centering
\caption{Dynamic discrepancy at \(T_*=1.5\) with second-order continuum spatial operators. The last column is the largest temporal discrepancy among the discrete, leading, and corrected runs when \(\Delta T\) is replaced by \(2\Delta T\).}
\label{tab:dynamic-convergence}
\begin{tabular}{rccccc}
\hline
\(N\) & leading \(E_\infty\) & rate & corrected \(E_\infty\) & rate & time check\\
\hline
64   & \(1.972\times10^{-2}\) & --    & \(1.712\times10^{-4}\) & --    & \(3.99\times10^{-10}\)\\
128  & \(9.587\times10^{-3}\) & 1.040 & \(4.338\times10^{-5}\) & 1.981 & \(2.52\times10^{-11}\)\\
256  & \(4.726\times10^{-3}\) & 1.020 & \(1.092\times10^{-5}\) & 1.990 & \(1.57\times10^{-12}\)\\
512  & \(2.346\times10^{-3}\) & 1.010 & \(2.739\times10^{-6}\) & 1.995 & \(1.03\times10^{-13}\)\\
1024 & \(1.169\times10^{-3}\) & 1.005 & \(6.860\times10^{-7}\) & 1.997 & \(1.73\times10^{-14}\)\\
\hline
\end{tabular}
\end{table}

\begin{figure}[htbp]
\centering
\includegraphics[width=0.67\textwidth]{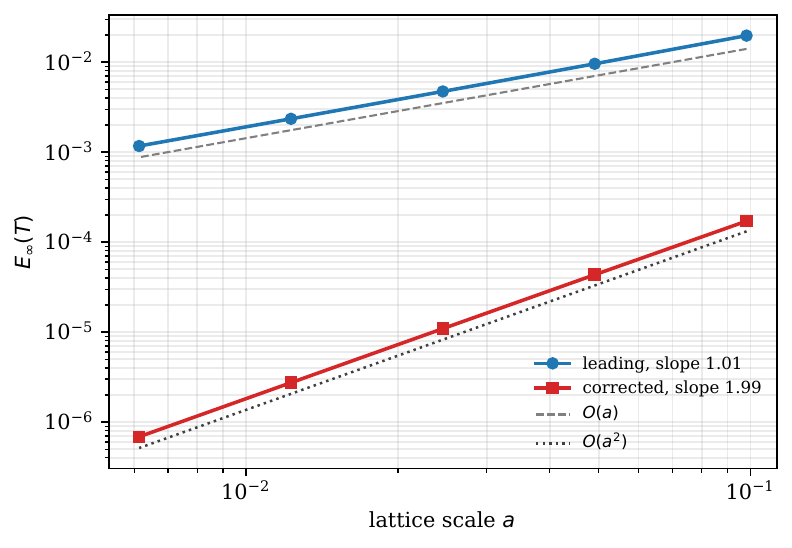}
\caption{Dynamic model discrepancy for the common data \eqref{eq:num-initial-data}. The discrepancies decrease at first- and second-order rates, respectively, over the tested range.}
\label{fig:dynamic-convergence}
\end{figure}
\FloatBarrier
The operators in \eqref{eq:numerical-difference-operators} are second order in space, as is the corrected discrepancy in Table~\ref{tab:dynamic-convergence}. We therefore repeat the continuum integrations with fourth-order centered derivatives, leaving the exact lattice equation, initial data, time steps, and norm unchanged. The alternative operators are
\be
\label{eq:num-fourth-order-operators}
\begin{aligned}
D_{1,4}U_j&=\frac{U_{j-2}-8U_{j-1}+8U_{j+1}-U_{j+2}}{12a},\\
D_{2,4}U_j&=\frac{-U_{j-2}+16U_{j-1}-30U_j+16U_{j+1}-U_{j+2}}{12a^2}.
\end{aligned}
\ee
Table~\ref{tab:spatial-order-check} gives the resulting discrepancies. The corrected fourth-order comparison retains an approximately quadratic rate, but its magnitude differs from the second-order comparison. At \(N=1024\), the two corrected continuum discretizations differ by \(1.017\times10^{-6}\), exceeding the \(6.860\times10^{-7}\) corrected discrepancy in Table~\ref{tab:dynamic-convergence}. Thus the latter is a comparison with a second-order semidiscrete continuum model, not an isolated error estimate against the continuum PDE.
\begin{table}[htbp]
\centering
\caption{Spatial-order sensitivity at \(T_*=1.5\). The first two error columns compare the discrete chain with the leading and corrected models using \eqref{eq:num-fourth-order-operators}; the last compares the corrected second- and fourth-order continuum discretizations.}
\label{tab:spatial-order-check}
\begin{tabular}{rccc}
\hline
\(N\) & leading, fourth order & corrected, fourth order & corrected spatial change\\
\hline
64   & \(1.997\times10^{-2}\) & \(4.214\times10^{-4}\) & \(2.621\times10^{-4}\)\\
128  & \(9.651\times10^{-3}\) & \(1.058\times10^{-4}\) & \(6.565\times10^{-5}\)\\
256  & \(4.742\times10^{-3}\) & \(2.644\times10^{-5}\) & \(1.635\times10^{-5}\)\\
512  & \(2.350\times10^{-3}\) & \(6.607\times10^{-6}\) & \(4.076\times10^{-6}\)\\
1024 & \(1.170\times10^{-3}\) & \(1.651\times10^{-6}\) & \(1.017\times10^{-6}\)\\
\hline
\end{tabular}
\end{table}
\FloatBarrier

For these smooth data and this fixed time interval, the two continuum evolutions remain close to the discrete trajectory, and their formal orders are observed at the solution level. This is numerical evidence only; it does not prove a stability estimate or a discrete-to-continuum error bound.
The fourth-order check confirms the rates but shows that the corrected error magnitude in Table~\ref{tab:dynamic-convergence} is sensitive to the spatial scheme.

}

\phantomsection
\addcontentsline{toc}{section}{Acknowledgments}
\subsubsection*{Acknowledgments}
The first author (DK) was supported by an NSF postdoctoral fellowship, MS 1502852, at the time most of this project was completed. He is currently supported by NSF DMS grant 2554417.
Most of this research was done when the second author (PLF) was a visiting research fellow at the Courant Institute for Mathematical Sciences, New York University. He is currently supported by the French project ANR-23-CE40-0010-02 and the European project MSCA-101131233.

\phantomsection 
\subsubsection*{Note on AI use}
After preparing a preliminary version of this manuscript, we used OpenAI's ChatGPT 5.6 to improve the flow of the exposition, check spelling and grammar, and write the Python code for Section~\ref{sec:numerics}. 


\bibliographystyle{plain}
\bibliography{DK-PLF}

\end{document}